\documentclass[a4paper, 12pt]{amsart}

\usepackage{cite}

\pdfoutput=1
\usepackage{latexsym, amsmath, amssymb, amsthm, mathrsfs, bm, mathtools,comment}

\usepackage[mathscr]{eucal}
\usepackage{subfigure}

\usepackage{array}
\usepackage{tikz}
\usetikzlibrary{cd}
\usepackage{aliascnt}
\usepackage{graphicx}
\usepackage[T1]{fontenc}
\usepackage{mathptmx}
\usepackage{microtype}
\usepackage{enumerate, enumitem}
\usepackage[colorlinks=true,linkcolor=blue,urlcolor=blue]{hyperref}

\usepackage[onehalfspacing]{setspace}

\usepackage[inner=2.4cm,outer=2.4cm, bottom=3.2cm]{geometry}

\DeclareMathAlphabet{\altmathcal}{OMS}{cmsy}{m}{n}

\allowdisplaybreaks

\usepackage{xcolor, color, soul}

\usepackage{color}
\usepackage[all]{xy}
\usepackage{enumerate}
\usepackage{mathbbol,mathtools}

\newtheorem{theorem}{Theorem}[section]

\newtheorem{headthm}{Theorem}

\newaliascnt{headcor}{headthm}
\newtheorem{headcor}[headcor]{Corollary}
\aliascntresetthe{headcor}

\newaliascnt{headconj}{headthm}

\aliascntresetthe{headconj}

\newaliascnt{corollary}{theorem}
\newtheorem{corollary}[corollary]{Corollary}
\aliascntresetthe{corollary}

\newaliascnt{claim}{theorem}

\aliascntresetthe{claim}

\newaliascnt{lemma}{theorem}
\newtheorem{lemma}[lemma]{Lemma}
\aliascntresetthe{lemma}

\newaliascnt{conjecture}{theorem}

\aliascntresetthe{conjecture}

\newaliascnt{proposition}{theorem}
\newtheorem{proposition}[proposition]{Proposition}
\aliascntresetthe{proposition}

\theoremstyle{definition}
\newaliascnt{definition}{theorem}

\aliascntresetthe{definition}

\newaliascnt{notation}{theorem}

\aliascntresetthe{notation}

\newaliascnt{example}{theorem}
\newtheorem{example}[example]{Example}
\aliascntresetthe{example}

\newaliascnt{examples}{theorem}

\aliascntresetthe{examples}

\newaliascnt{remark}{theorem}
\newtheorem{remark}[remark]{Remark}
\aliascntresetthe{remark}

\newaliascnt{fact}{theorem}

\aliascntresetthe{fact}

\newaliascnt{question}{theorem}

\aliascntresetthe{question}

\newaliascnt{questions}{theorem}

\aliascntresetthe{questions}

\newaliascnt{problem}{theorem}

\aliascntresetthe{problem}

\newaliascnt{construction}{theorem}

\aliascntresetthe{construction}

\newaliascnt{setup}{theorem}

\aliascntresetthe{setup}

\newaliascnt{algorithm}{theorem}

\aliascntresetthe{algorithm}

\newaliascnt{observation}{theorem}

\aliascntresetthe{observation}

\newaliascnt{discussion}{theorem}

\aliascntresetthe{discussion}

\newaliascnt{defprop}{theorem}

\aliascntresetthe{defprop}

\def\equationautorefname~#1\null{(#1)\null}
\def\sectionautorefname~#1\null{Section #1\null}
\def\subsectionautorefname~#1\null{\S #1\null}

\def\trdeg{{\rm trdeg}}

\definecolor{myorange}{RGB}{255, 160, 70}

\usepackage[normalem]{ulem}
\newcommand{\stkout}[1]{\ifmmode\text{\sout{\ensuremath{#1}}}\else\sout{#1}\fi}

\def \depth{{\operatorname{depth\, }}}

\def \Spec{{\operatorname{Spec}}}

\def \Proj{{\operatorname{Proj\, }}}

\def \trdeg{{\operatorname{trdeg}}}

\def \ord{{\operatorname{ord}}}

\def \codim{{\operatorname{codim}}}
\def \divv{{\operatorname{div}}}

\DeclareMathOperator{\Div}{Div}

\def \f1{\mathbf{1}}

\newcommand{\kk}{\mathbb{k}}

\def\xi{x}

\def\ls{\leqslant}
\def\gs{\geqslant}

\def\fm{\mathfrak{m}}

\def\fp{\mathfrak{p}}
\def\fq{\mathfrak{q}}

\def\fq{\mathfrak{q}}

\def\fm{\mathfrak{m}}

\def\m{\mathfrak{m}}

\def \PP{\mathbb P}

\def \QQ{\mathbb Q}

\def \NN{ {\mathbb Z}_{\gs 0}}

\def \ZZ{\mathbb Z}

\def \L{\mathcal L}

\begin{document}

\title[Multiplicities and degree functions  in local rings via intersection products]{Multiplicities and degree functions  in local rings via intersection products}

\author[Steven Dale Cutkosky]{Steven Dale Cutkosky$^1$}
\address{Steven Dale Cutkosky\\
	Department of Mathematics, University of Missouri, Columbia,
	MO 65211, USA. \emph{Email:} {\rm cutkoskys@malnourished}}
\thanks{$^{1}$ The first author was partially funded by NSF Grant DMS \#2348849.}

\author[Jonathan Monta{\~n}o]{Jonathan Monta{\~n}o$^2$}
\address{Jonathan Monta{\~n}o\\School of Mathematical and Statistical Sciences, Arizona State University, P.O. Box 871804, Tempe, AZ 85287-18041, USA. \emph{Email:} {\rm montano@asu.edu}}
\thanks{$^{2}$ The second author was partially funded by NSF Grant DMS \#2401522.}

\begin{abstract}
 We prove a theorem on the intersection theory over a Noetherian local ring $R$, which gives a new proof of a classical theorem of Rees about degree functions. To obtain this, we define an intersection product on schemes that are proper and birational over such rings  $R$, using the theory of rational equivalence developed by Thorup, and the Snapper-Mumford-Kleiman intersection theory for proper schemes over an Artinian local ring. Our development of this product is essentially self-contained. As a central component of the proof of our main theorem, we extend to arbitrary Noetherian local rings a formula by Ramanujam that computes Hilbert-Samuel multiplicities. In the final section, we express mixed multiplicities  in terms of intersection theory and conclude from this that they satisfy a certain multilinearity condition. Then we interpret some theorems of Rees and Sharp and of Teissier about mixed multiplicities over  $2$-dimensional excellent local rings in terms of our intersection product.
\end{abstract}

\keywords{}
\subjclass[2020]{Primary: 13H15, 14C17.}

\maketitle

\section{Introduction}
In this article we consider an arbitrary $d$-dimensional (Noetherian) local ring $(R,\fm_R,\kk)$, and interpret Hilbert-Samuel multiplicities, degree functions, and 
mixed multiplicities of $\fm_R$-primary ideals in $R$ via intersection products on suitable birational proper $\Spec(R)$-schemes $Y$. We give new geometric proofs of some classical theorems in commutative algebra about multiplicities as corollaries of  more general theorems  on the  intersection theory of exceptional Cartier divisors on  such schemes  $Y$.

  We define in  \autoref{SecIntProd} an intersection product $(-)_R$ on schemes that are proper and birational over   $R$. 
For this, we use the  theory of rational equivalence of Thorup   for finite type schemes over a Noetherian base \cite{Thor}. We also  adapt  some of the material from the intersection theory over fields in 
Fulton's book \cite{Fl} which extends to our more general setting. This allows us to   reduce to the case of proper schemes over a field (the residue field of $R$), so that we may  use the intersection theory of Snapper \cite{Sn}, Mumford \cite{Mum}, and Kleiman \cite{Kl}, \cite{Kl2} to make  computations. We  discuss intersection products and the contents of  \autoref{SecIntProd} in more detail at the end of this introduction.

 The (Hilbert-Samuel) multiplicity $e(I)$ of an $\fm_R$-primary ideal  $I$ is defined to be 
$$
e(I)=\lim_{n\rightarrow \infty}d!\frac{\ell_R(R/I^n)}{n^d}
$$
where $\ell_R(-)$ is the length of an $R$-module. In \cite{Ra} Ramanujam obtained a formula that computes multiplicities through intersection products. Different versions of this formula have appeared in the literature since then under various particular assumptions.  In \mbox{\cite{L}} the formula is stated for the case that 
$R$ is  the local ring of  a closed point on an algebraic scheme over an algebraically closed field. In \mbox{\cite{BLQ}} the formula is proven when $R$ is an analytically irreducible domain and $Y$ is normal.   Related formulas are also found in \mbox{\cite{KT}}, \cite{KV}, \cite{GGV}, and \cite{CR}.  In our first result  we extend this theorem to the full generality of   Noetherian local rings.
\begin{headthm}[{\autoref{TheoremMultInt}}]
Let $(R,\fm_R,\kk)$ be a $d$-dimensional local ring  and $I\subset R$ be an $\m_R$-primary ideal. 
 Let $\pi:Y\rightarrow \Spec(R)$ be a birational projective morphism such that $I\mathcal O_Y$ is invertible. Then 
\begin{equation*}
e(I)=-((I\mathcal O_Y)^d)_R.
\end{equation*}
\end{headthm}

The starting point of this paper was the realization that it was possible to use  intersection theory to give a simple and intuitive geometric proof of a celebrated theorem of Rees on degree functions.
Rees'  theorem is as follows:

\begin{theorem}[{Rees, \cite[\S 2]{Re}, \cite[Theorem 9.31]{Re2}}]\label{ReesThm*'} Let $(R,\fm_R,\kk)$ be an analytically unramified  local  domain   and $I\subset R$ be an $\fm_R$-primary ideal. Then the degree function $d_I: R\setminus\{0\}\to \NN$ given by $d_I(x)=e(I(R/xR))$ for $0\ne x$ can be expressed as
\begin{equation}\label{eq:Rees_Thm'}
d_I(x)=\sum_{i=1}^rd_i(I)v_i(x),
\end{equation}
where $v_1,\ldots, v_r$ are the Rees valuations of $I$ and $d_i(I)\in \ZZ_{\gs 0}$ for all $i$.
\end{theorem}
Rees valuations are defined and their theory developed in \cite{R5} and \cite[\S 10]{HS}. Rees cites Samuel's paper \cite{Sam} for the beginning of the theory of degree functions.

We  prove the following theorem about intersection products from which we deduce in \autoref{CorThm2} a geometric proof of \autoref{ReesThm*'}.

\begin{headthm}[{\autoref{Theorem2}}]\label{Theorem2'} Let   $(R,\fm_R,\kk)$ be an  analytically unramified local domain  of dimension $d$. Let $I$ be an $\fm_R$-primary ideal of $R$ and  $\pi:Y=\Proj(\oplus_{n\gs 0}\overline{I^n})\rightarrow\Spec(R)$ be the natural projective morphism.  Suppose that $F_1,\ldots, F_{d-1}$ are Cartier divisors on $Y$ with support above $\m_R$.  Then for any $0\ne x\in R$ we have
$$
-(F_1\cdot\ldots\cdot F_{d-1}\cdot T^*)_R=\sum_{E}v_E(x)(F_1\cdot\ldots
\cdot F_{d-1}\cdot E)_{R},
$$
where $T^*$ is the strict transform of $T=\Spec(R/xR)$ in $Y$ and
 $E$ ranges over the integral components of $\pi^{-1}(\fm_R)$. The local rings $\mathcal O_{Y,E}$ are discrete valuation rings and the canonical valuations $v_E$ of these $E$ are the Rees valuations of $I$. 
\end{headthm}

As a corollary, we have the following geometric interpretation of  degree functions.

\begin{headcor}[{\autoref{CorThm2}}]\label{CorThm2'}  Let   $(R,\fm_R,\kk)$ be an  analytically unramified local domain  of dimension $d$.  Let $I$ be an $\fm_R$-primary ideal of $R$ and $\pi:Y=\Proj(\oplus_{n\gs 0}\overline{I^n})\rightarrow\Spec(R)$ be the natural projective morphism
 and let $\mathcal L_Y =I\mathcal O_Y$.  Then for any $0\ne x\in R$ we have
$$
e(I(R/xR))=\sum_{E} v_E(x)(\mathcal L_Y^{d-1}\cdot E)_{R},
$$
where  $E$ ranges over the  integral components of $\pi^{-1}(\fm_R)$.  The local rings $\mathcal O_{Y,E}$ are discrete valuation rings and the canonical valuations $v_E$ of these $E$ are the Rees valuations of $I$. 
\end{headcor}

The coefficients $d_i(I)$ in  \autoref{ReesThm*'}
are uniquely determined by \cite[Theorem 9.42]{Re2}, so we have, by  \autoref{CorThm2'}, the geometric interpretation that the coefficients in \autoref{eq:Rees_Thm'} are equal to 
\begin{equation}\label{I6}
d_i(I)=((I\mathcal O_Y)^{d-1}\cdot E_i)_R,
\end{equation}
where $v_i=v_{E_i}$ for all $i$ are the Rees valuations of $I$.

 In \cite{Re},  and more precisely in  \cite[Theorem 9.41]{Re2}, Rees deduces other more general theorems whose proofs rely on  \autoref{ReesThm*'} as an essential ingredient; we state these theorems in    \autoref{genReesThm*} and \autoref{mostgenReesThm*} of this paper, and give short proofs in  \autoref{SecLR}.

We also obtain the following geometric interpretation of Rees' theorem \autoref{mostgenReesThm*}, which extends  \autoref{ReesThm*'} to more general rings. We note that if   $R$ is an excellent local ring, then the assumptions of the next theorem are satisfied.

\begin{headthm}[{\autoref{ThmExc}}]\label{ThmExc'} 
 Let  $(R,\fm_R,\kk)$ be a $d$-dimensional  local ring such that $R/P$ is analytically unramified  for every  prime ideal $P$ such that $\dim(R/P)=d$. Let $I$ be is an $\fm_R$-primary ideal such that $\oplus_{n\gs 0}\overline {I^n}$ is a finitely generated $R$-algebra.
Let $\pi:\Proj(\oplus_{n\gs 0}\overline{I^n})\rightarrow \Spec(R)$ be the natural projective morphism and set $\mathcal L_Y=I\mathcal O_Y$.
Then for every  $x\in R$ that is a nonzero divisor such that $\dim (R/xR)=d-1$ we have
$$
e(I(R/xR))=\sum_va_v(\mathcal L_Y^{d-1}\cdot C(Y,v))_Rv(x),
$$
where the sum is over the Rees valuations v of $I$. Here $C(Y,v)$ is the center of $v$ on $Y$, i.e.,  the integral subscheme  of $Y$ that is the closure of the unique point $q$ of $Y$ such that the local ring $\mathcal O_{Y,q}$ is dominated by the valuation ring  of $v$ and   $a_v=\ell_{R_{P_v}}(R_{P_v})$
with $P_v=\{x\in R\mid v(x)=\infty\}$.
\end{headthm}

In \autoref{SecMix} we turn to mixed multiplicities. Let $I_1,\ldots,I_r$ be $\fm_R$-primary ideals in $R$. For $n_1,\ldots,n_r\in \NN$, we have the following polynomial expansion
$$
e(I_1^{n_1}\cdots I_r^{n_r}) = \sum_{\stackrel{v_1,\ldots, v_r\in \NN,}{v_1+\cdots+v_r=d}}\binom{d}{v_1,\ldots, v_r}e(I_1^{[v_1]},\ldots, I_r^{[v_r]})n_1^{v_1}\cdots n_r^{v_r},
$$
see \autoref{prop:mixed_as_HS}. The coefficients $e(I_1^{[v_1]},\ldots, I_r^{[v_r]})$ are the {\it mixed multiplicities} of  $I_1,\ldots,I_r$ (Teissier and Risler \cite{T}, \cite[Definition 17.4.3]{HS}). In our next result, we prove that mixed multiplicities 
 can be interpreted as intersection products. This result is proven in \cite{Laz1} under the assumption that $R$ is essentially of finite type over a field. Related results also appear  in \cite{KT}.
 \begin{headthm}[{\autoref{thm:MixedMult}}]\label{thm:MixedMult'}
Let $(R,\fm_R,\kk)$ be a $d$-dimensional local ring  and   $I_1,\ldots, I_r$ be $\fm_R$-primary ideals of $R$. Let  $\pi:Y\rightarrow \Spec(R)$ be a proper birational morphism such that $I_i\mathcal O_Y$ is locally principal for every $i$. 
 Then for every $v_1,\ldots, v_r\in \NN$ such that $v_1+\cdots+v_r=d$ we have
$$ 
e(I_1^{[v_1]},\ldots, I_r^{[v_r]})=-\big((I_1\mathcal O_Y)^{v_1}\cdot\ldots\cdot(I_r\mathcal O_Y)^{v_r}\big)_R.
$$
 \end{headthm}

As a consequence of \autoref{thm:MixedMult'}, we obtain in  \autoref{cor:Ident_Mixed_M}
the following identity for mixed multiplicities  in an arbitrary local ring. Under more assumptions on $R$, a related formula appeared in  \cite[Theorem 8.1]{KK}  
$$ 
e\big((IJ)^{[v_1]},I_2^{[v_2]}\ldots, I_r^{[v_r]}\big)=
\sum_{i=0}^{v_1}\binom{v_i}{i}
e\big(I^{[i]},J^{[v_1-i]},I_2^{[v_2]},\ldots, I_r^{[v_r]}\big).
$$

In the final subsection \autoref{SecSemi}, we give geometric proofs of some classical theorems in multiplicity theory for 2-dimensional local rings.  
Our proofs are  based on  the theorem of Mumford and Lipman which states that the intersection matrix of a resolution of singularities of a surface singularity is negative definite \cite{Mu2}, \cite{L}. 
In \autoref{ineq2}, we use this theorem to prove that 
if $R$ is a 2-dimensional excellent  local ring  and $I,J$ are $\fm_R$-primary ideals, then
$$
e(I,J)^2\ls e(I)e(J).
$$
This inequality was originally proven by Teissier in \cite{ELT} for 2-dimensional complex analytic local rings, also using the negative definiteness of the intersection matrix of a resolution of surface singularities. 
The formula was later proven for arbitrary 2-dimensional local rings in \cite{ReSh} and \cite[Lemma 17.7.1]{HS}, using different methods. In \autoref{thm:2dim_pos_def} we give a geometric proof of a theorem of Rees and Sharp on 2-dimensional Cohen-Macaulay rings \cite[Theorem 5.4]{ReSh}. Because of our geometric methods, we must restrict to 2-dimensional excellent local rings, but we do not assume the ring to be Cohen-Macaulay.

We now discuss in more detail  \autoref{SecIntProd} on intersection theory over local rings. In this section,  we give a  self-contained development of the intersection theory that we  use.

In \cite[\S 20]{Fl}, Fulton defines the relative dimension of a finite type scheme over a regular base, and  constructs a theory of rational equivalence for such schemes using this definition of relative dimension. He points out that much of the material from earlier parts of his book on the intersection theory of schemes of finite type over a field with the ordinary definition of dimension extends to this relative setting.  In  \cite[\S 2]{Kl3}, Kleiman gives a modified definition of relative dimension, which allows him to construct rational equivalence over a Noetherian base that is equidimensional and universally catenary. Finally, in \cite{Thor}, Thorup gives a definition of  relative dimension that allows him to construct rational equivalence for finite type schemes  over an arbitrary Noetherian base.

In  \autoref{SecCycle}, we  introduce the notation about cycles on a Noetherian scheme. We survey the theory of rational equivalence developed  in \cite{Thor} over $S=\Spec(R)$, where $R$ is a local ring, and explain how some material from the first two chapters of \cite{Fl} extends to this relative situation.  We also discuss the degree, $\int(-)$, defined in \cite[1.4]{Fl} for schemes that are proper over a field. 
This operation is only defined for a cycle of dimension 0 over a field as base scheme, but we can reduce to this case if we are working with a cycle  over $S$ that is  supported over $\fm_R$. In this case,  the cycle can be viewed as being a cycle over the field $\kk=R/\fm_R$. Thus  if the cycle has dimension 0 as a cycle over $\kk$, the degree $\int(-)$ can be computed over  $\kk$. We define in  \autoref{SecLRInt}  the intersection product $(-)_R$ discussed above and  prove a few propositions that are needed in the rest of the article. 
In  \autoref{SecIPPS}, we begin by summarizing  the intersection product of Snapper \cite{Sn}, Mumford \cite{Mum}, and Kleiman  \cite{Kl}, \cite{Kl2} for schemes that are proper over an Artinian base.  We then give a self-contained proof in \autoref{equiv} that the intersection product $(-)_R$ defined in  \autoref{SecLRInt} can be computed using the Snapper-Mumford-Kleiman intersection theory. This reduction is essential for our applications of intersection theory to the study of multiplicities and degree functions.

Throughout this  paper, all local rings $(R,\fm_R,\kk)$ are assumed to be Noetherian. If $R$ is a  ring, $\ell_R(M)$ denotes the length of an $R$-module $M$.

\section{Intersection theory over local rings}\label{SecIntProd}

 In this section we define   an intersection product $(-)_R$ on schemes which are proper and birational over  a local ring $R$. This  theory  generalizes others in  the literature, see \cite{L} and \cite[\S 2.3.2]{BLQ}. We use the  theory of rational equivalence of Thorup  \cite{Thor} for finite type schemes over a noetherian base. Intersection theory  derived from \cite{Thor} is used in \cite{KT}.

 If $X$ is a Noetherian scheme and $V\subset X$ is an integral subscheme, then $R(V)$  denotes  the {\it function field of $V$}, i.e., $R(V)=\mathcal O_{X,V}/\fm_V$ where $\mathcal O_{X,V}$ is the local ring of the generic point of $V$ and $\fm_V$ is its maximal ideal.

\subsection{Cycles}\label{SecCycle}
In this section let $X$  be a Noetherian scheme.  We introduce notation which is compatible with notation in  the first chapter of \cite{Fl}, but in the generality of Noetherian schemes, without the restriction of being an algebraic scheme over a field, or even assuming that the schemes are universally catenary. We use this to have a uniform notation when combining results from different sources.

Define $Z(X)$ to be the free abelian group generated by $[V]$, where $V$ is an integral closed subscheme of $X$. 
If $\mathcal F$ is a coherent sheaf on $X$, define  the fundamental cycle of $\mathcal F$ as
$$
[\mathcal F]=\sum a_i[V_i],
$$
where $V_i$ are the integral components of the support of $\mathcal F$ and $a_i=\ell_{\mathcal O_{X,V_i}}(\mathcal F_{V_i})$.  For $Z$ a closed subscheme of $X$, define $[Z]=[\mathcal O_Z]$. 

Suppose that $X$ is an integral Noetherian scheme and $V$ is a closed integral subscheme of $X$ with $\dim(\mathcal O_{X,V})=1$. Let $r\in R(V)$.  Define the cycle $\ord_V(r)$ as   in \cite[\S 1.2]{Fl}, as we now summarize. Let $A=\mathcal O_{X,V}$, which is a one-dimensional  local ring. Write $r=\frac{a}{b}$ with $a,b\in A$. Then 
$$
\ord_V(r)=\ell_A(A/aA)-\ell_A(A/bA).
$$

The fundamental cycle $[\divv(r)]\in Z(X)$ is defined as 
$$
[\divv(r)]=\sum\ord_{V}(r)[V],
$$
where the sum is over the closed integral subschemes $V$ of $X$ which satisfy $\dim(\mathcal O_{X,V})=1$. Note that if $X$ is not  universally catenary, then it is possible that $\dim(V)<\dim(X)-1$ for some of the $V$ for which $\ord_V(r)\neq 0$.

Let $X$ be a scheme and $\mathcal K_X$ be the sheaf of $\mathcal O_X$-modules such that $\Gamma(U,\mathcal K_X)$ is the total quotient ring of $\Gamma(U,\mathcal O_X)$ for all affine open subsets $U$ of $X$. 
Let $\mathcal K_X^*\subset \mathcal K_X$ be the subsheaf of units of $\mathcal K_X$, i.e., $\Gamma(U,\mathcal K_X^*)$ is the multiplicative group of invertible elements of $\Gamma(U,\mathcal K_X)$. A Cartier divisor $D$ on $X$ is determined by the data $(U_i,f_i)$, where  $\{U_i\}$ is an open cover of $X$ and $f_i\in \Gamma(U_i,\mathcal K_X^*)$ are such that $\frac{f_i}{f_j}$ is a unit in $\Gamma(U_i\cap U_j,\mathcal O_X)$. The invertible sheaf $\mathcal O_X(D)$ associated to $D$ is defined by $\mathcal O_X(D)|U_i=\frac{1}{f_i}\mathcal O_X|U_i$. We can add Cartier divisor $D$ and $E$ so that $\mathcal O_X(D+E)\cong\mathcal O_X(D)\otimes\mathcal O_X(E)$ and $\mathcal O_X(-D)\cong \mathcal O_X(D)^{-1}$.  A Cartier divisor is defined to be {\it effective} if $\mathcal O_X(-D)$ is a sheaf of ideals.

Let $D$ be a Cartier divisor on a Noetherian integral scheme $X$ with $(U_i,f_i)$ as above. One defines $\ord_V(D)=\ord_V(f_i)$, where $U_i$ is an open set that contains the generic point of $V$ (as  in \cite[\S 2.1]{Fl}). Furthermore, the  {\it associated Weil divisor} of $D$ is defined as

\begin{equation}\label{Weil}
[D]=\sum \ord_V(D)[V]\in Z(X),
\end{equation}
where the sum is over the closed integral subschemes $V$ of $X$ which  satisfy $\dim(\mathcal O_{X,V})=1$. 

Let $f:X\rightarrow Y$ be a proper morphism of Noetherian schemes and $V$ be a closed integral subscheme of $X$. Then the image $W=f(V)$ is a closed integral subscheme of $Y$.
Define
$$
\deg(V/W)=
\begin{cases}
[R(V):R(W)]&\mbox{ if $[R(V):R(W)]$ is finite}\\
0&\mbox{ otherwise.}
\end{cases}
$$
Define $f_*[V]=\deg(V/W)[W]$, which extends linearly to a homomorphism $f_*:Z(X)\rightarrow Z(Y)$.

If $\alpha\in Z(X)$, we   write $|\alpha|$ for the support of $\alpha$. We also  write $|D|$ for the support of a Cartier divisor $D$, which is the union of the $V$ with nonzero $\ord_V(D)$ appearing in the expansion \autoref{Weil}  (as in \cite{Fl}). We can regard the support of a cycle or Cartier divisor on $X$ as a closed subscheme of $X$ by giving it the induced reduced scheme structure.

Let $X$ and $Y$ be schemes with finitely many irreducible components. A morphism $\alpha:X\rightarrow Y$ is {\it birational} if $\alpha$  induces a bijection between the sets of irreducible components of $X$  and $Y$,  and for every generic point $\eta\in X$ of an irreducible component of $X$, the induced local homomorphism $\mathcal O_{Y,\alpha(\eta)}\rightarrow \mathcal O_{X,\eta}$ is an isomorphism. If $\alpha:X\rightarrow Y$ is a birational morphism and $[\mathcal O_X]=\sum m_i[V_i]$, then $[\mathcal O_Y]=\sum m_i[\alpha(V_i)]$.

\begin{lemma}\label{Lembirat} Let $A$ be a Noetherian ring and $K\subset A$ be an ideal. Let 
$$
\pi:Y=\Proj(\oplus_{n\gs 0}K^n)\rightarrow \Spec(A)
$$
be the blowup of $K$. Then
\begin{enumerate}
\item[\rm (1)] $\pi$ is a projective morphism. 
\item[\rm (2)] $\pi$ is birational if and only if $K$ is not contained in any minimal prime of $A$.
\end{enumerate}
\end{lemma}

\begin{proof} Let $f_1,\ldots,f_n$ be generators of  $K$.   We have a natural graded surjection of $A$-algebras $A[t_1,\ldots,t_n]\twoheadrightarrow \oplus_{n\gs 0}K^n$, where the $t_i$ are variables. This gives  a closed immersion $Y\subset \PP_A^{n-1}$, showing $\pi$ is projective. 
For the proof of part (2), suppose $K\subset P$ where $P$ is a minimal prime of $A$. Then we have an induced morphism 
$$
Y\times_{\Spec(A)}\Spec(A_P)\cong\Proj(\oplus_{n\gs 0}K_P^n)\rightarrow \Spec(A_P).
$$
Since every element of $K_P$ is nilpotent, we have
$$
\Proj(\oplus_{n\gs 0}K_P^n)=\{\mbox{homogeneous primes $Q$ in $\oplus_{n\gs 0}K_P^n$ such that $\oplus_{n>0}K_P^n\not\subset Q$}\}=\emptyset.
$$
In particular, there does not exist a point in $Y$ that  maps to $P$ and so $\pi$ is not birational. Conversely, assume that $K$ is not contained in any minimal prime  of $A$. Let $U=\Spec(A)\setminus V(K)$, an open subset of $\Spec(A)$. The induced morphism $\pi:\pi^{-1}(U)\rightarrow U$ is an isomorphism by  \cite[Proposition II.7.13]{H}. By assumption, $U$ contains all minimal primes of $A$ so $\pi$ is birational.
\end{proof}

\subsection{Intersection Theory over $\Spec(R)$}\label{SecRelInt}

 In this section, we recall some results about the construction of rational equivalence over Noetherian schemes from \cite{Thor}. The article \cite{Thor} gives a generalization of the development of rational equivalence for algebraic schemes over a field  in \cite[Ch. 1]{Fl}.

Throughout  this section, let $(R,\fm_R,\kk)$ be a $d$-dimensional  local ring  and let 
 $S=\Spec(R)$.
Let $\pi:X\rightarrow\Spec(R)$ be a  finite type morphism.
For an  integral subscheme $V$ of $X$ with image $T=\pi(V)$ in $S$ define the {\it relative dimension}
$$
\dim_SV=\trdeg(R(V)/R(T))-\dim(\mathcal O_{S,T}).
$$

The following lemma relates the relative dimension of $V$ with that of a   subvariety.
\begin{lemma}\label{lem:rel_dim}
If $W\subset V$ are integral subschemes of $X$, then 
\begin{equation}\label{eqReldim}
\dim _SW\ls\dim_SV-\dim (\mathcal O_{V,W}).
\end{equation}
In particular, if $W\ne V$ then $\dim_SW<\dim_SV$.
\end{lemma}

\begin{proof} Let $U=\pi(W)$ and $T=\pi(V)$. By formula \cite[IV.5.6.5.1]{EGAIV},
$$
\trdeg(R(V)/R(T))+\dim(\mathcal O_{T,U})\gs \trdeg(R(W)/R(U))
+\dim(\mathcal O_{V,W}),
$$
so $\dim_SV-\dim(\mathcal O_{V,W})\gs \dim_SW$,
since $\dim(\mathcal O_{S,U})\gs \dim(\mathcal O_{S,T})+\dim( \mathcal O_{T,U})$.
\end{proof}

Relative dimensions attain only a bounded range of nonnegative values, as the following proposition shows.

\begin{proposition}\label{prop:bound_rel_dim}
Let    $\pi:X\rightarrow S$ be proper and  birational and $V$ be an integral closed subscheme of $X$. Then $-d\ls \dim_SV\ls 0$. Moreover, $\dim_SV=0$ if and only if $V$ is an integral component of $X$, and $\dim_SV=-d$ if and only if $V$ is a closed point of $X$, and so  is contained in $\pi^{-1}(\m_R)$.
\end{proposition}

\begin{proof} If $V$ is an integral component of $X$ then $\dim_SV=0$ since $\pi$ is birational, by the definition of relative dimension. If $W$ is a closed integral subscheme of $X$, then $W\subset V$ for some integral component $V$ of $X$.  Thus $\dim_SW\ls 0$ by \autoref{eqReldim}, and $\dim_SW=0$ if and only if $W$ is an integral component of $X$. 
We have that $W\cap \pi^{-1}(\m_R)\ne\emptyset$ since $\pi$ is proper, so there exists a closed point $p\in W$ and so $-d=\dim_Sp\ls\dim_SW$ by \autoref{lem:rel_dim}.
\end{proof}

For $V$ an integral subscheme of $X$, define the {\it relative codimension}
$$
\codim_S(V,X)=\dim_SX-\dim_SV.
$$
For a cycle $\alpha$ on $X$, let $\alpha_i$ denote the homogeneous component of relative dimension $\dim_S$ equal to $i$ and $\alpha^i$ denote the homogeneous component of $\codim_S(-,X)$ equal to $i$. Let $Z_i(X/S)$ be the cycles which are homogeneous of $\dim_S$ equal to $i$ and $Z^i(X/S)$ be the cycles which are homogeneous of $\codim_S(-,X)$ equal to $i$, giving two gradings of $Z(X)$, $Z_*(X/S)=\oplus Z_i(X/S)$ and $Z^*(X/S)=\oplus Z^i(X/S)$. These two gradings are related by $Z^i(X/S)=Z_{\dim_SX-i}(X/S)$.

Let $W$ be a closed  integral subscheme of $X$. A cycle $\alpha$ on $W$ is called {\it principal} if there exists $0\ne r\in R(W)$ such that 
$$
\alpha = [\divv(r)]^1=\sum \ord_V(r)[V],
$$
where the sum is over the integral subschemes $V$ of $W$ such that $\codim_S(V,W)=1$.

Define $P^1(W/S)$ to be the subgroup of $Z^1(W/S)$ which is generated by the principal cycles on $W$. We have that
$Z^1(W/S)=Z_{\dim_SW-1}(W/S)\subset Z_{\dim_SW-1}(X/S)$, so that $P^1(W/S)$ can be considered as a subgroup of $Z_{\dim_SW-1}(X/S)$.
In \cite[Definition 6.3]{Thor},
 $P_*(X/S)$ is defined to be the subgroup generated by principal cycles on the integral subschemes of $X$, giving $P_*(X/S)$ a natural grading $P_*(X/S)\cong \oplus P_i(X/S)$
 by $\dim_S$. Cycles $\alpha$ and $\alpha'$ on $X$ are said to be {\it rationally equivalent} if $\alpha-\alpha'\in P_*(X/S)$. The quotient $A_*(X/S)=Z_*(X/S)/P_*(X/S)$ is the group of cycles modulo rational equivalence. Moreover, $A_*(X/S)\cong\oplus A_i(X/S)$ is graded by $\dim_S$ so that  $A_i(X/S)=Z_i(X/S)/P_i(X/S)$.

 We state the following proposition from \cite{Thor}, which gives a generalization of \cite[Theorem 1.4]{Fl}. 
 
 \begin{proposition}[{\cite[Proposition 6.5]{Thor}}]\label{PropRE}
Let $Y$ and $X$ be finite type $S$-schemes  and $f:Y\rightarrow X$ be a proper $S$-morphism.  Then the homomorphism $f_*:Z(Y)\rightarrow Z(X)$ defined in  \autoref{SecCycle} induces a functorial homomorphism
 $f_*:A_k(X/S)\rightarrow A_k(Y/S)$ for all  $k$.
 \end{proposition}
 
In the previous proposition observe that by definition, if $V$ is an integral subscheme of $Y$ and $W=f(V)$, then $f_*[V]=0$ unless $[R(V):R(W)]$ is finite. In the latter case we also have that  
 $\dim_SW=\dim_SV$, since letting $U$ be the common image of $V$ and $W$ in $S$,
 $$
 \trdeg(R(V)/R(U))=\trdeg(R(V)/R(W))+\trdeg(R(W)/R(U))
 =\trdeg(R(W)/R(U)).
 $$
 
 \begin{remark} When $R=\kk$ is a field, we  write $A_*(X/\kk)$ for $A_*(X/S)$. In this case, $A_t(X/\kk)=A_t(X)$ as defined in \cite[\S 1.1]{Fl}, since $\dim _S$ is just  the ordinary dimension of a scheme in this case.
 \end{remark}

 Let $D$ be a Cartier divisor on $X$, and $V$ be a closed integral subscheme of $X$, with inclusion $j:V\rightarrow X$.  As in \cite[Definitions 2.2.2 \& 2.3]{Fl}, we define  
\begin{equation}\label{pbCart}
 \mbox{a Cartier divisor }j^*D\mbox{ on }V
 \end{equation}
  with support  contained in $V\cap |D|$, and whose class $[j^*D]$ is well-defined in $A_{\dim_SV-1}(V\cap |D|/S)$  as follows. If $V\not \subset |D|$, then $D$ naturally extends to a Cartier divisor $j^*D$  as we now describe.
  Suppose that $D$ is represented by $(U_i,f_i)$. Since $V$ is not in the support of $D$, $f_{\alpha}$ naturally restricts to an element $\overline f_i\in R(V)$. Then $j^*D$ is  the Cartier divisor $(U_i\cap V,  \overline f_i)$. If $V$ is contained in $|D|$, then we choose any Cartier divisor $C$ on $V$ such that $\mathcal O_V(C)\cong \mathcal O_Y(D)\otimes\mathcal O_V$ and let $j^*D=C$.
 The cycle class $[j^*D]$ is well-defined. Indeed, suppose  that $V\subset |D|$ and $C_1, C_2$ are Cartier divisors on $V$ such that $\mathcal O_V(C_1)\cong \mathcal O_X(D)\otimes\mathcal O_V\cong \mathcal O_V(C_2)$. Then $[C_1]-[C_2]=[\divv(r)]$ for some $r\in R(V)$, so $[C_1]^1-[C_2]^1=[\divv(r)]^1=0$ in $A_{\dim_SV-1}(V\cap |D|/S)$.

 With some modifications, the construction of $D\cdot [V]$ in \cite[Definition 2.3]{Fl} extends to our more general situation. Let $k=\dim_SV$. 
 We 
 define  
 \begin{equation}\label{CartInt}
 D\cdot [V]= [j^*D]^1=\sum \ord_W(j^*D)[W]\in A_{\dim_SV-1}(|D|\cap V/S),
  \end{equation}
 where the sum is over the closed integral subschemes of $V$ which satisfy $\mbox{codim}_S(W,V)=1$. 
 When $S$ is universally catenary, such as when $R=\kk$ is a field this construction is exactly the one of \cite[Definition 2.3]{Fl}.

 The conclusions of \cite[Proposition 2.3 \& Corollary 2.4.2]{Fl}  naturally extend to 
 the relative situation of this section, working in $A_*(-/S)$. We  make use of these results.  For the reader's convenience, we now state some conclusions of \cite[Proposition 2.3]{Fl} in our relative situation.  

\begin{proposition}[{the relative form of \cite[Proposition 2.3]{Fl}}]\label{Prop2.3} Let $X$ be a finite type $S$-scheme.
\begin{enumerate}
\item[{\rm (a)}] Let $D$ be a Cartier divisor on $X$ and $\alpha,\alpha'\in Z_k(X/S)$. Then
$$
D\cdot(\alpha+\alpha')=D\cdot\alpha+D\cdot\alpha'
\mbox{ in }A_{k-1}(|D|\cap(|\alpha|\cup|\alpha'|)/S).
$$
\item[{\rm (b)}] If $D,D'$ are Cartier divisors on $X$, and $\alpha\in Z_k(X/S)$, then
$$
(D+D')\cdot\alpha=D\cdot\alpha+D'\cdot\alpha\mbox{ in }
A_{k-1}((|D|\cup |D'|)\cap |\alpha|)/S).
$$
\item[{\rm (c)}](Projection Formula) Let $D$ be a Cartier divisor on $X$, $f:X'\rightarrow X$ be a proper morphism, $\alpha\in Z_k(X/S)$ and $g$ the morphism from $f^{-1}(|D|)\cap|\alpha|$ to $|D|\cap f(|\alpha|)$ induced by $f$. Then
$$
g_*(f^*D\cdot\alpha)=D\cdot f_*(\alpha)\mbox{ in } A_{k-1}(|D|\cap f(|\alpha|)/S).
$$

\end{enumerate}

\end{proposition}

 Suppose that $i:V\rightarrow W$ is an inclusion of closed subschemes of $X$. Then by 
  \autoref{PropRE},
  we have a natural homomorphism $i_*:A_*(V/S)\rightarrow A_*(W/S)$. We will regard elements of $A_*(V/S)$ as elements of $A_*(W/S)$ when it is convenient. For computation of intersection products, there is no loss of necessary information  by doing this.  
 
 When $S=\kk$ is a field and $X$ is a complete $\kk$-scheme, we may use the construction of \cite[Definition 1.4]{Fl}, which is as follows.   For a zero-cycle $\alpha=\sum n_P[P]\in A_{0}(X/\kk)$ define
\begin{equation}\label{Req13}
\deg(\alpha) = \int_X\alpha =\sum_Pn_P[R(P):\kk].
\end{equation}
  This definition depends on the base field $\kk$. 
If   $\pi:X\rightarrow \Spec(\kk)$ is the structure morphism and $\pi_*(\alpha)=n[\kk]\in A_0(\kk)=\ZZ[\kk]$, then equivalently $\deg(\alpha)=\int_X\alpha=n$. 
As explained in the discussion of \cite[Definition 1.4]{Fl}, if $f:X\rightarrow Y$ is a proper morphism of complete $\kk$-schemes, then 
\begin{equation}\label{Req16}
\int_X\alpha=\int_Yf_*\alpha.
\end{equation}

Suppose that $\kk\rightarrow \kk'$ is a field extension, which induces a morphism $\iota:\Spec(\kk')\rightarrow \Spec(\kk)$. For 
 $\alpha=n[\kk']\in A_0(\kk')=A_0(X/\kk)$, we have $\iota_*\alpha=n[\kk':\kk][\kk]\in A_0(\kk)$. Therefore, if $\pi':X\rightarrow \Spec(\kk')$ is  the structure morphism of a $\kk'$-scheme $X$, then $(i\pi')_*\alpha=i_*\pi_*'\alpha=n[\kk':\kk][\kk]$. So the computation of $\int_X\alpha$ over $\kk$ is equal to $[\kk':\kk]$ times the one  of $\int_X\alpha$ over $\kk'$.

We  will  make use of the following computation when $X$ is proper over $S$.
Let $Z\subset X_{\kk}=X\times_S\Spec(R/\fm_R)$ be a closed subscheme. Then
\begin{equation}\label{changebase}
A_{-d}(Z/S)=A_0(Z/\kk),
\end{equation}
 and for $\alpha\in A_0(Z/\kk)$, $\int_{Z}\alpha=\int_{X_{\kk}}\alpha$,  computed over $\kk$.

\subsection{Intersection products over a local ring}\label{SecLRInt}
 Let $(R,\fm_R,\kk)$ be a $d$-dimensional local ring. Let $S=\Spec(R)$  and  $\pi:Y\rightarrow S$ be a birational proper morphism. 
Suppose that $k\in \NN$, $\alpha\in A_k(Y/S)$  and $F_1,\ldots, F_t$ are Cartier divisors on $Y$ such that $t\gs k+d$, and at least one of the $F_i$ or $\alpha$ is supported above $\fm_R$.  Then we define
\begin{equation}\label{Proddef1}
(F_1\cdot\ldots\cdot F_{t}\cdot \alpha)_R=
\begin{cases}
\int_{Y_{\kk}}F_1\cdot\ldots\cdot F_{t}\cdot  [\alpha]&\mbox{ if }t=k+d\\
0&\mbox{ if }t>k+d.
\end{cases}
\end{equation}
Here we are using the intersection theory of \autoref{SecRelInt}, and computing the product over $\kk$.  We first calculate the cycle class
$$
F_1\cdot\ldots\cdot F_{t}\cdot [\alpha]\in A_{-d}(Y_{\kk}/S),
$$
 by iterating the operation of \autoref{CartInt}. Here $Y_{\kk}=Y\times_S\Spec(R/\fm_R)$ is a proper $\kk$-scheme and 
 $$
 A_{-d}(Y_{\kk}/S)=A_0(Y_{\kk}/\kk)
 $$
  by \autoref{changebase}, so we may apply formula \autoref{Req13}.
  
  If $F_1,\ldots, F_d$ are Cartier divisors on $Y$ such that at least one of the $F_i$ are supported above $\fm_R$, we define
  \begin{equation}\label{Proddef2}
  (F_1\cdot\ldots \cdot F_d)_R=\int_{Y_{\kk}}F_1\cdot\ldots\cdot F_d\cdot [Y],
\end{equation}
where we compute the product over $\kk$. By  \autoref{prop:bound_rel_dim}, the class of the cycle  $[Y]$ is  in $A_0(Y/S)$, so 
$F_1\cdot\ldots\cdot F_d\cdot [Y]\in A_{-d}(Y_{\kk}/S)=A_0(Y_{\kk}/\kk)$.

Because every invertible sheaf on an integral scheme is isomorphic to the sheaf of a Cartier divisor, we may intersect with invertible sheaves also.

We now compile some useful propositions about this intersection product. 
We first observe that the intersection product over a local ring is multilinear and symmetric.
 
\begin{proposition}\label{CorR18} 
 Let   $(R,\fm_R,\kk)$ be a  $d$-dimensional   local ring   and $Y\rightarrow \Spec(R)$ be a proper and birational morphism.  
Suppose that $k\in \NN$ and $\alpha\in A_*(Y/S)$ is a cycle such that $[\alpha]_i=0$ for $i>k$, and $F_1,\ldots, F_t$ are Cartier divisors on $Y$ such that $t= k+d$, and at least one of the $F_i$ or $\alpha$ is supported above $\fm_R$. Then 
$(F_1\cdot\ldots\cdot F_{d-1}\cdot \alpha)_R$ is symmetric in the $F_i$ and is multilinear in the $F_i$ and $\alpha$.

Furthermore, suppose that $F_1,\ldots, F_d$ are Cartier divisors on $Y$ with support above $\m_R$. Then
$(F_1\cdot\ldots\cdot F_d)_R$ is multilinear and symmetric in the $F_i$.
\end{proposition}

\begin{proof}  The conclusions follow from the definition of the intersection product and  parts (a) and (b) of  \autoref{Prop2.3}. 
Commutativity follows from \cite[Corollary 2.4.2]{Fl}, which is  valid in $A_*(Y/S)$.
\end{proof}

Next we show that the intersection product over local rings is invariant under proper and birational morphisms.

\begin{proposition}\label{ThmR1}
 Let   $(R,\fm_R,\kk)$ be a  $d$-dimensional   local ring   and 
 $Y\xrightarrow{f} X\xrightarrow{\pi} \Spec(R)$ be proper and birational morphisms.  Suppose that $F_1,\ldots,F_d$ are Cartier divisors on $X$ such that at least one of the $F_i$ is supported above $\fm_R$.  Then
$$
(f^*F_1\cdot\ldots\cdot f^*F_d)_R=(F_1\cdot\ldots\cdot  F_d)_R.
$$
If $\alpha\in A_k(Y/S)$ and $t\gs k+d$ then
$$
(f^*F_1\cdot\ldots\cdot F_t\cdot \alpha)_R=(F_1\cdot\ldots\cdot F_t\cdot f_*\alpha)_R.
$$
\end{proposition}

\begin{proof} 
 Let $h$ be the morphism from $Y_{\kk}$ to  $X_{\kk}$ induced by $f$. 
Since $f$ is birational, $f_*[Y]=[X]$ in $A_0(X/S)$. Thus, since $|F_i|\subset X_{\kk}$ for some $i$,  we have by part (c) of  \autoref{Prop2.3} that
\begin{equation}\label{Req17}
h_*(f^*F_1\cdot\ldots\cdot f^*F_d\cdot [Y])=F_1\cdot\ldots\cdot F_d\cdot [X]
\end{equation}
in $A_{-d}(X_{\kk}/S)$. Thus
\begin{align*}
(f^*F_1\cdot\ldots\cdot f^*F_d)_R=\int_{Y_{\kk}}f^*F_1\cdot\ldots\cdot f^*F_d\cdot[Y]
&=\int_{X_{\kk}}h_*(f^*F_1\cdot\ldots\cdot f^* F_{d}\cdot[Y])\\
&= \int_{X_{\kk}}F_1\cdot\ldots\cdot F_d\cdot[X]=(F_1\cdot\ldots\cdot F_d)_R,
\end{align*}
where the  second equality is by   \autoref{Req16}, the third one from \autoref{Req17}.

The second formula is proven similarly.
\end{proof}

The following proposition offers a summation formula of the intersection product in terms of the highest dimensional prime ideals of the underlying local ring. 

\begin{proposition}\label{AssocForm}
Let $(R,\fm_R,\kk)$ be a $d$-dimensional  local ring and $\{P_1,\ldots, P_t\}$  be the set of minimal prime ideals $P_i$  of $R$ such that $\dim(R/P_i)=\dim(R)$.  Suppose that $\pi:Y\rightarrow \Spec(R)$ is proper and birational and $F_1,\ldots, F_d$ are Cartier divisors on $Y$ such that at least one of them has  support above $\m_R$. For $1\ls i\ls t$, let  $Y_i$ be the integral component of $Y$ such that $Y_i$ dominates $R/P_i$ and  let  $i_{Y_i}:Y_i\rightarrow Y$ be the natural closed immersion. Then
$$
(F_1\cdot\ldots\cdot F_d)_R=
\sum_{i=1}^t \ell_{R_{P_i}}(R_{P_i})(i_{Y_i}^*F_1\cdot\ldots\cdot i_{Y_i}^*F_d)_{R/P_i}.
$$
\end{proposition}

\begin{proof} Let $\{P_{t+1},\ldots,P_s\}$ be the minimal prime ideals of $R$ such that $\dim R/P_i<d$ for $t<i\ls s$. Clearly, this set might be empty, in which case we take $s=t$. For $i>t$, let $Y_i$ be the integral component of $Y$ such that $Y_i$ dominates $R/P_i$.
We have that  
$$
[Y]=\sum_{i=1}^s a_i[Y_i]
\quad  \text{where} 
\quad
a_i=\ell_{\mathcal O_{Y,Y_i}}(\mathcal O_{Y,Y_i})=\ell_{R_{P_i}}(R_{P_i}).
$$
Thus 
\begin{align*}
(F_1\cdot\ldots\cdot F_d)_R=\int_{Y_{\kk}}F_1\cdot\ldots\cdot F_d\cdot[Y]
&=\sum_{i=1}^s a_i\int_{Y_{\kk}}F_1\cdot\ldots\cdot F_d\cdot [Y_i]\\
&=\sum_{i=1}^s a_i\int_{(Y_i)_{\kk}}i_{Y_i}^*F_1\cdot\ldots\cdot i_{Y_i}^*F_d\cdot [Y_i]\\
&=\sum_{i=1}^s \ell_{R_{P_i}}(R_{P_i})(i_{Y_i}^*F_1\cdot\ldots\cdot i_{Y_i}^*F_d)_{R/P_i}.
\end{align*}
By \autoref{prop:bound_rel_dim}, $\dim_SY_i=0$ for each of the components $Y_i$; that is, 
$[Y_i]\in A_0(Y/S)$ for $1\ls i\ls s$. If $F_1\cdot \ldots\cdot F_d\cdot[Y_i]\ne 0$ in $A_{-d}((Y_i)_{\kk})$, then there exists a chain of distinct integral schemes 
$$
W_1\subset W_2\subset\cdots\subset W_d\subset Y_i,
$$
where each $W_j$ is a components of the cycle $F_j\cdot\ldots\cdot F_d\cdot [Y_i]$ for $1\ls j\ls d$. Thus the  ordinary  dimension of the scheme $Y_i$ is at least $d$, and so $i\ls t$. Thus 
$\int_{Y_{\kk}}F_1\cdot\ldots\cdot F_d\cdot [Y_i]=0$ if $i>t$ and so 
$$
(F_1\cdot\ldots\cdot F_d)_R=\sum_{i=1}^t \ell_{R_{P_i}}(R_{P_i})(i_{Y_i}^*F_1\cdot\ldots\cdot i_{Y_i}^*F_d)_{R/P_i}
$$
obtaining the desired conclusion.
\end{proof}

We end the section with a formula that gives a relation for the  intersection products  in a finite  local ring extension.

 \begin{proposition}\label{Theoremnormal*} Let $(R,\fm_R,\kk)$ be a $d$-dimensional  local domain.  
Let $T$ be a domain that is a finite extension of $R$ and  $\fm_1,\ldots,\fm_r$ be its maximal ideals. 
Set $T_i=T_{\m_i}$ for every $i$. 
Let $Y\rightarrow \Spec(R)$ be a birational projective morphism which is
the blowup of an ideal $K$ of $R$. Let $W$ be the blowup of $KT$ and  $\alpha:W\rightarrow Y$ be the induced morphism. 
Let 
$W_i=W\times_{\Spec(T)}\Spec(T_i)$. 
Let $F_1,\ldots,F_d$ be Cartier divisors on $Y$  such that at least one of them has  support above $\fm_R$. 
Then 
\begin{equation*}\label{eqI3}
[T:R](F_1\cdot\ldots\cdot F_d)_R=\sum_{i=1}^r[T/\m_{i}:R/\m_R](\alpha^*F_1\cdot\ldots\cdot \alpha^*F_d)_{T_i},
\end{equation*}
where $[T:R]$ is the degree of the extension of quotient fields of $T$ and $R$.
\end{proposition}

\begin{proof} 
   Let $[W]$ be the class of $W$ in $A_0(W/S)$ and $[Y]$ be the class of $Y$ in $A_0(Y/S)$. Then $\alpha_*[W]=[T:R][Y]$ ($\alpha_*$ is defined in  \autoref{SecCycle}). Let  $h:W_{\kk}\rightarrow Y_{\kk}$ be the morphism induced by $\alpha$.   Then by part (c) of  \autoref{Prop2.3},
$$
h_*(\alpha^*F_1\cdot\ldots\cdot \alpha^*F_d\cdot [W])=F_1\cdot\ldots\cdot F_d\cdot\alpha_*[W]=[T:R](F_1\cdot\ldots\cdot F_d\cdot[Y])
$$
in $A_{-d}(Y_{\kk}/S)=A_0(Y_{\kk}/\kk)$. 

Let $\phi:W\rightarrow \Spec(T)$ be the induced morphism. We have that $|\phi^{-1}(\fm_i)|=|(W_i)_{\kk}|$ for all $i$ and $|W_{\kk}|$ is the disjoint union of the $|\phi^{-1}(\fm_i)|$ for $1\ls i\ls r$. 
Computing over $\kk$, 
\begin{align*}
[T:R]\int_{Y_{\kk}}F_1\cdot\ldots\cdot F_d\cdot[Y]&=\int_{W_{\kk}}\alpha^*F_1\cdot\ldots\cdot \alpha^*F_d\cdot[W]\\
&=\sum_i\int_{|\phi^{-1}(\m_i)|}\alpha^*F_1\cdot\ldots\cdot \alpha^*F_d\cdot[W]\\
&=\sum_i\int_{(W_i)_{\kk}}\alpha^*F_1\cdot\ldots\cdot \alpha^* F_d\cdot[W_i].
\end{align*}
 As explained after \autoref{Req16}, if $\kk_i=T/\fm_i$, then the computation of $\int_{(W_i)_{\kk}}\alpha^*F_1\cdot\ldots\cdot \alpha^*F_d\cdot[W_i]$ over $\kk$ is equal to $[\kk_i:\kk]$ times its calculation over $\kk_i$. Finally,   we have 
 $$
[T:R](F_1\cdot\ldots\cdot F_d)_R=\sum_i[\kk_i:\kk](\alpha^*F_1\cdot\ldots\cdot \alpha^*F_d)_{T_i}.
$$
\end{proof}

\subsection{Intersection theory on proper schemes over an Artinian ring}\label{SecIPPS} In this subsection we show that the intersection product introduced in \autoref{SecRelInt} can be computed using Snapper-Mumford-Kleiman intersection theory.

Let $X$ be a $d$-dimensional proper scheme 
over  an Artinian  ring $\kk$ (note $\kk$ is not necessarily a field here), $\mathcal L_1,\ldots,\mathcal L_t$ be invertible sheaves on $X$, and
$\mathcal F$ be a coherent sheaf on $X$ of dimension $\ls t$. 
The {\it intersection product} 
$$
I(\mathcal L_1\cdot\ldots\cdot\mathcal L_t;\mathcal F)_{\kk}=I(\mathcal L_1\cdot\ldots\cdot\mathcal L_t;\mathcal F)_{\kk}
$$
is the coefficient of $n_1\cdots n_t$ in the monic numerical polynomial
\begin{equation*}
\chi_{\kk}(\mathcal L_1^{n_1}\otimes\cdots\otimes\mathcal L_t^{n_t}\otimes \mathcal F)
=\sum_{i=0}^d(-1)^i \ell_\kk \left(H^i(X,\mathcal L_1^{n_1}\otimes\cdots\otimes\mathcal L_t^{n_t}\otimes \mathcal F)\right),
\end{equation*}
where $\chi_{\kk}(-)$ denotes the Euler 
characteristic. When the base ring $\kk$ is understood, we will write $I(\mathcal L_1\cdot\ldots\cdot \mathcal L_t;\mathcal F)$.

The definition of the intersection product is given in \cite[Definition I.2.1]{Kl} and  \cite[Definition B.8]{Kl2}.  As in \cite{Kl}, we write $I(\mathcal L_1\cdot\ldots\cdot\mathcal L_t)$ or $I(\mathcal L_1\cdot\ldots\cdot \mathcal L_t\cdot X)$ for $I(\mathcal L_1\cdot\ldots\cdot\mathcal L_t;\mathcal O_X)$. 
 We refer the reader to \cite[\S I.2]{Kl} and  \cite[Appendix B]{Kl2} for details of  important properties of this intersection product. 
 Although \cite{Kl} assumes that $\kk$ is  an algebraically closed field, this assumption is not necessary for the results 
 that we use from there.

We now show that the intersection product of \autoref{SecLRInt}  can be computed using the intersection product of this section. 

\begin{theorem}\label{equiv}
Let  $(R,\m_R,\kk)$ be a $d$-dimensional  local ring. Let $\pi:Y\rightarrow \Spec(R)$ be a birational proper morphism and   $[Y]=\sum a_jY_j$.
Let $E_1,\ldots, E_r$ be the irreducible components of $\pi^{-1}(\m_R)$ of dimension $d-1$.   
Let $F_1,\ldots,F_d$ be Cartier divisors on $Y$ such that $F_d$ has  support above   $\m_R$. 
Each $F_i$ restricts naturally to a Cartier divisor $F_i|Y_j$ on every $Y_j$.
Then
\begin{equation}\label{LocIntgen}
(F_1\cdot\ldots\cdot F_d)_R=\sum_{i}  m_i 
I\big((\mathcal O_{Y}(F_1)\otimes\mathcal O_{E_i}\cdot\ldots \cdot \mathcal O_{Y}(F_{d-1})\otimes\mathcal O_{E_i};\mathcal O_{E_i}\big)_{\kk},
\end{equation}
where 
$m_i=\sum_ja_j\ord_{E_i}(F_d|Y_j)$. Here we set $\ord_{E_i}(F_d|Y_j)=0$ if $E_i\not\subset Y_j$. 
\end{theorem}

\begin{proof} 
Let $S=\Spec(R)$. We have that $\dim_SY_i=0$ for all $i$ by  \autoref{prop:bound_rel_dim}. If $V\subset \pi^{-1}(\fm_R)$ is closed integral subscheme then $\dim_SV=-1$ if and only if $\trdeg(R(V)/\kk)=d-1$, and the latter holds if and only if $\dim(V)=d-1$ (the ordinary dimension of $V$), since $V$ is of finite type over $\kk$. Thus, from formula \autoref{CartInt}, 
$$
F_d\cdot [Y_i]=\sum_j\ord_{E_j}(F_d|Y_i)[E_j]\in A_{-1}(Y_{\kk}/S)
$$
and so
\begin{align*}
(F_1\cdot\ldots\cdot F_d)_R&=\sum_ja_i\int_{Y_{\kk}}F_1\cdot\ldots\cdot F_d\cdot [Y_i]
=\sum_ia_i\big(\sum_j\ord_{E_j}(F_d|Y_i)\int_{Y_{\kk}}F_1\cdot\ldots\cdot F_{d-1}\cdot [E_j]\big)\\
&=\sum m_i\int_{E_i}F_1\cdot\ldots\cdot F_{d-1}\cdot [E_i].
\end{align*}
We have thus reduced to showing that 
$$
I(\mathcal O_Y(F_1)\otimes\mathcal O_{E_i}\cdot \ldots\cdot \mathcal O_Y(F_{d-1})\otimes\mathcal O_{E_i})_{\kk}
=\int_{E_i}F_1\cdot\ldots\cdot F_{d-1}\cdot [E_i]
$$
for all $i$.

Let $\iota: E_i\rightarrow Y$ be the inclusion and define Cartier divisors $\iota^*F_j$ for $1\ls j\ls d-1$ in $A_{-2}(E_i/S)=A_{d-1}(E_i/\kk)$ by \autoref{pbCart}. By part (c) of  \autoref{Prop2.3},
$$
\iota_*(\iota^*F_1\cdot\ldots\cdot \iota^*F_{d-1}\cdot [E_i])=F_1\cdot\ldots\cdot F_{d-1}\cdot [E_i]
$$
in $A_0(E_i/\kk)$, 
so by \autoref{Req16},
$$
\int_{E_i}\iota^*F_1\cdot\ldots\cdot \iota^*F_{d-1}\cdot [E_i]=\int_{E_i}F_1\cdot\ldots\cdot F_{d-1}\cdot [E_i].
$$ 
Since $\mathcal O_Y(F_j)\otimes\mathcal O_{E_i}\cong \mathcal O_{E_i}(\iota^*(F_j))$,
we have reduced to showing that
\begin{equation}\label{SameInt}
\int_{E_i}\iota^*F_1\cdot\ldots\cdot \iota^*F_{d-1}\cdot [E_i]=I(\mathcal O_{E_i}(\iota^*F_1)\cdot\ldots\cdot \mathcal O_{E_i}(\iota^*F_{d-1}))_{\kk}
\end{equation}
for all $i$. 

Since $E_i$ is a complete $\kk$-variety, by \cite[Corollary II.5.6.2]{EGAII} (Chow's lemma), there exists a birational projective morphism $\alpha:\overline E_i\rightarrow E_i$ such that $\overline E_i$ is a projective $\kk$-variety.
Now
$$
\int_{\overline E_i}\alpha^*\iota^*F_1\cdot\ldots\cdot \alpha^*\iota^*F_{d-1}\cdot [\overline E_i]
=\int_{E_i}\iota^*F_1\cdot\ldots\cdot \iota^*F_{d-1}\cdot [E_i]
$$
by part (c) of  \autoref{Prop2.3} and \autoref{Req16}, and
$$
I(\mathcal O_{\overline E_i}(\alpha^*\iota^*F_1)\cdot\ldots\cdot \mathcal O_{\overline E_i}(\alpha^*\iota^*F_{d-1}))_{\kk}
=
I(\mathcal O_{E_i}(\iota^*F_1)\cdot\ldots\cdot \mathcal O_{E_i}(\iota^*F_{d-1}))_{\kk}
$$
by \cite[Lemma B.15]{Kl2}. Thus we have reduced to showing that 
\begin{equation}\label{Projeq}
\int_EF_1\cdot\ldots\cdot F_{d-1}\cdot [E]=I(\mathcal O_E(F_1)\cdot\ldots\cdot \mathcal O_E(F_{d-1}))_{\kk},
\end{equation}
if $E$ is a $d-1$ dimensional projective $\kk$-variety and $F_1,\ldots, F_{d-1}$ are Cartier divisors on $E$.

Suppose that $W$ is a $d-k$-dimensional closed subvariety of $E$. Let $\iota:W\rightarrow E$ be the inclusion, and let $G_{d-k}=\iota^*F_{d-k}$ be defined by \autoref{pbCart}.
Thus as defined in \autoref{CartInt},
\begin{equation}\label{eq:cyc_Gd1}
F_{d-k}\cdot[W]=[G_{d-k}]=\sum \ord_{W_j}G_{d-k}[W_j]\in A_{d-k-1}(E/\kk).
\end{equation}

Since $W$ is a projective variety over $\kk$,  there exists a very ample divisor $H$ on $W$.  Therefore there exists $a>0$ and effective Cartier divisors $H_1$ and $H_2$ on $W$ such that
$H_2$ is linearly equivalent to $aH+G_{d-k}$ and $H_1$ is linearly equivalent to $aH$ on $W$. Now $G_{d-k}$ is linearly equivalent to $H_2-H_1$ so there exists $r\in R(W)^*$ such that $[\divv(r)]=[G_{d-k}]-([H_2]-[H_1])$. 
The definition of $\divv(r)$ is in  \autoref{SecCycle} and \cite[\S 1.3]{Fl}. We thus have 
$$
 F_{d-k}\cdot[W]=[G_{d-k}]=[H_2]-[H_1]
$$ in $A_{d-k}(E/\kk)$. 
In \cite[Definition B.1]{Kl2}, $K_r(E)$ is defined to be the subgroup of the Grothendieck group  $K(E)$ which is generated by the classes of coherent sheaves whose support has dimension less than or equal to $r$. By \cite[Example 1.6.5]{Fl}, there is a well defined homomorphism $A_{d-k}(E/\kk)\rightarrow K_{d-k}(E)/K_{d-k-1}(E)$ that takes the class of a $d-k$-dimensional subvariety $V$ of $E$ to the class of $\mathcal O_V$. Thus
\begin{equation}\label{LI}
[\mathcal O_{H_2}]-[\mathcal O_{H_1}]-\sum_j\ord_{W_j}G_{d-k}[\mathcal O_{W_j}]\in K_{d-3}(E),
\end{equation}  
and so
\begin{align*}
&I(\mathcal O_E(F_1)\otimes\mathcal O_{W}\cdot\ldots\cdot \mathcal O_{E}(F_{d-k})\otimes\mathcal O_{W})_{\kk}\\
&=I(\mathcal O_E(F_1)\otimes\mathcal O_{W}\cdot\ldots\cdot \mathcal O_{E}(F_{d-k-1})\otimes\mathcal O_{W}\cdot \mathcal O_{W}(H_2-H_1))_{\kk}\\
&=I(\mathcal O_E(F_1)\otimes\mathcal O_{W}\cdot\ldots\cdot \mathcal O_{E}(F_{d-k-1})\otimes\mathcal O_{W}\cdot \mathcal O_{W}(H_2))_{\kk}\\
&\,\,\,\,\,-I(\mathcal O_E(F_1)\otimes\mathcal O_{W}\cdot\ldots\cdot \mathcal O_{E}(F_{d-k-1})\otimes\mathcal O_{W}\cdot\mathcal O_{W}(H_1))_{\kk}\\
&=I(\mathcal O_E(F_1)\otimes\mathcal O_{W}\cdot\ldots\cdot \mathcal O_{E}(F_{d-k-1})\otimes\mathcal O_{W};\mathcal O_{H_2})_{\kk}\\
&\,\,\,\,\,-I(\mathcal O_E(F_1)\otimes\mathcal O_{W}\cdot\ldots\cdot \mathcal O_{E}(F_{d-k-1})\otimes\mathcal O_{W}; \mathcal O_{H_1})_{\kk}\\
&=\sum_j \ord_{W_j}G_{d-1}I(\mathcal O_E(F_1)\otimes\mathcal O_{W_j}\cdot\ldots\cdot \mathcal O_E(F_{d-k-1})\otimes \mathcal O_{W_j})_{\kk},
\end{align*}
where the third equality follows from  \cite[Lemma B.13]{Kl2} (or \cite[Proposition I.2.4 and Corollaries I.2.1 \& I.2.2]{Kl}) and the last equality from \autoref{LI} and \cite[Theorem B.9]{Kl2}.  
By \autoref{eq:cyc_Gd1}, we have  that
$$
\int_EF_1\cdot\ldots\cdot F_{d-k}\cdot [W]=\sum_j\ord_{W_j}G_{d-1}\int_EF_1\cdot\ldots\cdot F_{d-k-1}\cdot [W_j].
$$
By induction on $k$, we have reduced the proof of \autoref{Projeq} to showing that 
$$
I(R(P))_{\kk}=\int_{E}[P]=\deg(P)
$$
for $P$  a closed point on  $E$ and $\deg(P)$ as defined in \autoref{Req13}. By their definitions, both  $I(R(P))_{\kk}$ and $\deg(P)$ equal $[R(P):\kk]$. 
Thus the conclusion of the theorem follows.
\end{proof}

\begin{remark} A different proof of \autoref{equiv} can be obtained by applying\cite[Example 18.3.6]{Fl}, which states that
$$
\int_X(-)=I(-)_{\kk}
$$
on complete $\kk$-schemes $X$, as a consequence of  the Riemann-Roch theorem on algebraic schemes developed in \cite[\S 18.3]{Fl}. The proof of \autoref{equiv} thus follows from  \autoref{SameInt}  and by using 
  \cite[Example 18.3.6]{Fl} on the $E_i$.
\end{remark}

\begin{remark}\label{LocIntRmk}  If $F_d$ is an effective Cartier divisor, then, 
following the notation in \autoref{equiv}, we have 
$$
m_i=\ell_{\mathcal O_{Y,E_i}}(\mathcal O_{Y,E_i}/\mathcal O_Y(-F_d)_{E_i})
$$
for all $E_i$. Furthermore,
$$
(F_1\cdot\ldots\cdot F_d)_R=I(\mathcal O_Y(F_1) \cdot\ldots\cdot \mathcal O_Y(F_{d-1});\mathcal O_Y/\mathcal O_Y(-F_d))_{\Lambda}
$$
where $\Lambda$ is the Artinian local ring $\Lambda=R/\mathcal O_Y(-F_d)\cap R$.
\end{remark}
\begin{proof} Set $A=\mathcal O_{Y,E_i}$ with maximal ideal $\fm_A$ and let $f\in A$ be a local equation of $F_d$ in $A$.
The element  $f$ is a nonzero divisor of $A$ since $F_d$ is an effective Cartier divisor. If $f$ is a unit in $A$, then $f$ is a unit in $A/P$ for all minimal primes $P$ of $A$, so that $\ell_{\mathcal O_{Y,E_i}}(\mathcal O_{Y,E_i}/\mathcal O_Y(-F_d)_{E_i})=0$ and $\ord_{E_i}(F_d|Y_j)=0$ for all $j$.  Therefore 
$$
m_i=\ell_{\mathcal O_{Y,E_i}}(\mathcal O_{Y,E_i}/\mathcal O_Y(-F_d)E_i)=0.
$$
Thus we may assume that $f$ is not a unit in $A$.

Let $I=fA$, an $\m_A$-primary ideal. By the associativity formula \cite[Theorem 11.2.4]{HS},
$$
e(I)=\sum e(I(A/P))\ell_{A_P}(A_P)
$$
where the sum is over the minimal primes $P$ of $A$.  $f$ is a nonunit and nonzero divisor in $A$ and in $A/P$ for all $P$, so $\depth(A)>0$ and $\depth(A/P)>0$ for all $P$. Therefore $A$ and the $A/P$ are Cohen-Macaulay. 
Thus, by \cite[Proposition 11.1.10]{HS}
$$
e(I)=\ell_A(A/I)=\ell_{\mathcal O_{Y,E_i}}(\mathcal O_{Y,E_i}/\mathcal O_Y(-F_d)_{E_i})
$$
and if $A/P\cong \mathcal O_{Y_j,E_i}$, then
$$
e(I(A/P))=\ell_{A/P}((A/P)/f(A/P))=\ord_{E_i}(F_d\mid Y_j).
$$ 
Moreover, 
$a_j=\ell_{A_P}(A_P)$, giving the first formula in the statement. 

The second statement now follows since 
\begin{align*}
I(\mathcal O_Y(F_1)\cdot\ldots\cdot \mathcal O_Y(F_{d-1});\mathcal O_{Y}/\mathcal O_Y(-F_d))_{\Lambda}
&=\sum_i m_iI\big((\mathcal O_{Y}(F_1)\otimes\mathcal O_{E_i}\cdot\ldots\cdot \mathcal O_{Y}(F_{d-1})\otimes\mathcal O_{E_i};\mathcal O_{E_i}\big)_{\Lambda}\\
&=\sum_i m_iI\big((\mathcal O_{Y}(F_1)\otimes\mathcal O_{E_i}\cdot\ldots\cdot \mathcal O_{Y}(F_{d-1})\otimes\mathcal O_{E_i};\mathcal O_{E_i}\big)_{\kk}\\
&=(F_1\cdot\ldots\cdot F_d)_R,
\end{align*}
by \cite[Lemma B.12]{Kl2}, and since each $E_i$ is a $\kk$-variety and $\kk$ is a quotient of $\Lambda$.
\end{proof}

\section{Multiplicities of ideals as intersection products}\label{Mult}

This section is devoted to the proof  and some applications of a very general form of a  theorem of Ramanujam which expresses Hilbert-Samuel multiplicities in intersection-theoretic terms  \cite{Ra}.

Some other versions of this theorem are realized  in  \cite[\S1.6 B]{Laz1}, and more recently in \cite[Proposition 2.6]{BLQ}. Related formulas are proven in \cite{KT}.

\begin{theorem}\label{TheoremMultInt} 
Let $(R,\fm_R,\kk)$ be a $d$-dimensional local ring  and $I\subset R$ be an $\m_R$-primary ideal. Let $\pi:Y\rightarrow \Spec(R)$ be the blowup of $I$. The Hilbert-Samuel multiplicity  of $I$ is equal to 
$$ 
e(I)=-(\mathcal L_Y^d)_{R},
$$
where $\mathcal L_Y=I\mathcal O_Y$. More generally, let $\alpha:W\rightarrow \Spec(R)$ be a proper birational morphism such that $I\mathcal O_W$ is invertible. Then
$e(I)=-((I\mathcal O_W)^d)_R$.
\end{theorem}
\begin{proof} 
$Y=\Proj(\oplus_{n\gs 0}I^n)$ has the tautological very ample line bundle $\mathcal O_Y(1)=I\mathcal O_Y=\mathcal L_Y$.
Let $Z=\Proj(\oplus_{n\gs 0} I^n/I^{n+1})$. Then $Z$ is an effective  Cartier divisor on $Y$ with ideal sheaf $\mathcal O_Y(-Z)=I\mathcal O_Y=\mathcal L_Y$ and tautological line bundle $\mathcal O_Z(1)=\mathcal L_Y\otimes_{\mathcal O_Y}\mathcal O_Z$.
Moreover, $Z$ is a projective scheme over $\Lambda=R/I$.
For $n\gg0$ we have  $H^0(Z,\mathcal O_Z(n))=I^n/I^{n+1}$. 
This follows by the comment in the last line of the proof of  \cite[Theorem II.5.19]{H} which proves  \cite[Exercise II.5.9]{H}.  
We have that 
$$
\chi_{\Lambda}\big((\mathcal L_Y\otimes\mathcal O_Z)^n\big)_{\Lambda}
=\frac{I\big(({{\mathcal L}_Y}\otimes\mathcal O_Z\big)^{d-1})_{\Lambda}}{(d-1)!}n^{d-1}+Q(n),
$$
where $\chi_{\Lambda}$ is the euler characteristic over $\Lambda$ and $Q(n)$ is a polynomial of degree less than $d-1$ by \cite[Theorem 19.16]{C0},
which is valid over our Artinian local ring $\Lambda$, after substituting the reference to \cite[Theorem B.7]{Kl2} for \cite[Theorem 19.1]{C0}.
Since ${\mathcal L}_Y\otimes\mathcal O_Z$ is ample on $Z$ and using the above calculations, for $n\gg 0$ we have
\begin{align*}
\ell_R(I^n/I^{n+1})=\ell_R\big(H^0(Z,({\mathcal L}_Y\otimes\mathcal O_Z)^n)\big)
&=\chi_{\Lambda}\big(({\mathcal L}_Y\otimes\mathcal O_Z)^n\big)\\
&=\frac{I(({\mathcal L}_Y\otimes \mathcal O_Z)^{d-1})_{\Lambda}}{(d-1)!}n^{d-1}
+Q(n).
\end{align*}
Thus 
\begin{align*}
e(I)=\lim_{n\rightarrow\infty}\frac{\ell_R(I^n/I^{n+1})(d-1)!}{n^{d-1}}
&=I\big(({\mathcal L}_Y\otimes\mathcal O_Z)^{d-1}\big)_{\Lambda}\\
&=-(\mathcal L_Y^d)_R
\end{align*}
by  \autoref{LocIntRmk}, proving the first statement in the result.

For the second statement, we observe that since $I\mathcal O_W$ is invertible, there is a proper birational morphism $\lambda:W\rightarrow Y$ which factors $\alpha$, such that
 $\lambda^*(I\mathcal O_Y)\cong I\mathcal O_W$. Thus 
$$
e(I)=-((I\mathcal O_Y)^d)_R=-((I\mathcal O_W)^d)_R
$$
by the first part of this proof and \autoref{ThmR1}.
\end{proof}

As a corollary, we give a new proof of the associativity formula for multiplicities. A classical proof of this theorem can be found in  \cite[Theorem 11.2.4]{HS}.

\begin{corollary}[Associativity Formula] 
Let $(R,\fm_R,\kk)$ be a local ring  and $I\subset R$ be an $\m_R$-primary ideal.
Let $\Theta$ be the set of minimal prime ideals $P$ of $R$ such that $\dim(R/P)=\dim(R)$. Then
$$
e(I)=\sum_{P\in \Theta}\ell_{R_P}(R_{P})e(I(R/P)).
$$
\end{corollary}

\begin{proof} Let $\pi:Y\rightarrow\Spec(R)$ be the blowup of $I$. For $P\in \Theta$, let $Y_P$ be the integral component of $Y$ which dominates $\Spec(R/P)$. Then
$Y_P\rightarrow\Spec(R/P)$ is the blowup of $I(R/P)$ by \cite[Corollary II.7.15]{H}, so that
$$
e(I)=-((I\mathcal O_Y)^d)_R\quad\mbox{and}\quad
e(I(R/P))=-((I\mathcal O_{Y_P})^d)_{R/P}
$$
by  \autoref{TheoremMultInt}. The result now follows from  \autoref{AssocForm}.
\end{proof}

By combining \autoref{Theoremnormal*} and  \autoref{TheoremMultInt}, we obtain a proof of the following result; a classical proof  of it can be found in \cite[Corollary VIII.10.1]{ZS2}.

\begin{corollary}\label{Cornormal*} Let $(R,\fm_R,\kk)$ be a  local domain  and $I\subset R$ be an $\m_R$-primary ideal. 
Let $T$ be a domain that is a finite extension of $R$ and let $\m_1,\ldots,\m_r$ be its maximal ideals. 
Set $T_i=T_{\m_i}$ for every $i$. 
Let $Y\rightarrow \Spec(R)$ be the blowup of an ideal $K$ of $R$ such that $I\mathcal O_Y$ is invertible  and  $W$ be the blowup of $KT$, so that $I\mathcal O_W$ is invertible.  
Let 
$W_i=W\times_{\Spec(T)}\Spec(T_i)$. Then 
\begin{equation*}\label{eqI3}
[T:R]e(I)=\sum_{i=1}^r[T/\m_{i}:R/\m_R]e(IT_i),
\end{equation*}
where $[T:R]$ is the degree of the extension of quotient fields of $T$ and $R$.
\end{corollary}

\section{Degree functions}\label{SecDeg}

In this section we apply the intersection product developed in \autoref{SecIntProd} to the degree functions studied by Rees in \cite{Re}.

\subsection{Theorems of Rees on degree functions}

The following theorem is proven in  \cite[\S 2]{Re}. In the statement, there is a typo stating that $R$ is a local ring, when to be true it must be assumed that $R$ is a domain (it is a standing assumption in \cite[\S 2]{Re} that $R$ be an analytically unramified local domain). We have made this correction in the following statement. The necessity of the assumption that $R$ be a domain for the theorem to be true follows from the examples at the end of this subsection. 
The Rees valuations of an ideal are defined in \cite{R5} and in \cite[Ch. 10]{HS}.

\begin{theorem}[{\cite[Theorem 2.3]{Re}, \cite[Theorem 9.31]{Re2}}]\label{ReesThm*} Let $(R,\fm_R,\kk)$ be an analytically unramified  local  domain   and $I\subset R$ be an $\fm_R$-primary ideal. Then the degree function 
$d_I: R\setminus\{0\}\to \NN$ given by 
$d_I(x)=e(I(R/xR))$ for $0\ne x$ can be expressed as
\begin{equation}\label{eq:Rees_Thm}
d_I(x)=\sum_{i=1}^rd_i(I)v_i(x),
\end{equation}
where $v_1,\ldots, v_r$ are the Rees valuations of $I$ and $d_i(I)\in \NN$ for all $i$.
\end{theorem}

The proof of  \cite[Theorem 2.3]{Re} proves \autoref{ReesThm*}. In \cite[Theorem 9.31]{Re2}, we have the additional assumption that $R$ is formally equidimensional, giving the stronger conclusion that all $d_i(I)$ are positive in this case. 

The following  \autoref{ReesLemma} and  \autoref{genReesThm*} appear in \cite[\S 3]{Re}. There are also typos in the original statements in \cite{Re}. For  \autoref{ReesLemma} and  \autoref{genReesThm*} to be true, we must have the additional assumption that $x\in R$ is not a zero divisor. We have made this correction in the following statements. The necessity of the assumption that $x$ be a nonzero divisor for the lemma and theorem to be true follows from the examples at the end of this  subsection.

\begin{lemma}[{\cite[Lemma 3.1]{Re}}]\label{ReesLemma}
Let $(R,\fm_R,\kk)$ be a $d$-dimensional  local  ring   and $I\subset R$ be an $\fm_R$-primary ideal.
Suppose that $P_1,\ldots, P_s$ are the minimal prime ideals of $R$ such that $\dim(R/P_i)=d$.  Then  the degree function 
$d_I: R\setminus\{0\}\to \ZZ_{>0}$ given by $d_I(x)=e(I(R/xR))$ for $x\in R$ such that $\dim(R/xR)=d-1$ and $x$ is a nonzero divisor on $R$ can be expressed as
\begin{equation}\label{eqC2}
d_I(x)=\sum_{i=1}^s\ell_{R_{P_i}}(R_{P_i})d_{I(R/P_i)}(x).
\end{equation}
\end{lemma}

The proof of \cite[Lemma 3.1]{Re} proves  \autoref{ReesLemma}. The assumption that $x$ is a nonzero divisor is necessary to  conclude that the minimal primes $P$ of $R/xR$ of dimension $d-1$ contain some minimal prime  of $R$ of dimension $d$, and to  show that  the assumption of \cite[Theorem 4]{Nor} that $x$ is a nonzero divisor on $R_P$ is satisfied.

The following theorem depends on  \autoref{ReesLemma}.

\begin{theorem}
[{\cite[\S 3]{Re}}]
\label{genReesThm*}
Let $(R,\fm_R,\kk)$ be a $d$-dimensional  local  ring   and $I\subset R$ be an $\fm_R$-primary ideal. Then the degree function $d_I: R\setminus\{0\}\to \NN$ given by $d_I(x)=e(I(R/xR))$   for  $x\in R$ such that $\dim (R/xR)=d-1$
 and $x$ is a nonzero divisor on $R$ can be expressed as
\begin{equation}\label{eq:Rees_Thm}
d_I(x)=\sum_{i=1}^rd_i(I)v_i(x),
\end{equation}
where $v_1,\ldots, v_r$ are the Rees valuations of $I$ and $d_i(I)\in \NN$ for all $i$. 
\end{theorem}

In his later lecture notes \cite{Re2}, Rees modifies the definition of the degree function by adding a term which can only appear when $x\in R$ is  a zero divisor, and proves the generalization of  \autoref{genReesThm*} to the  situation where $x\in R$ is only assumed to satisfy $\dim(R/xR)=d-1$.

Suppose that $x\in R$. Then $(0:x)$ is an ideal in $R$ which is an $R/xR$-module. We may thus compute the multiplicity $e_{I(R/xR)}(0:x)$ of the $R/xR$-module $(0:x)$ with respect to the ideal $I(R/xR)$.

\begin{theorem}[{\cite[Theorem 9.41]{Re2}}]\label{mostgenReesThm*}
Let $(R,\fm_R,\kk)$ be a $d$-dimensional  local  ring   and $I\subset R$ be an $\fm_R$-primary ideal. Then the degree function $d_I: R\setminus\{0\}\to \NN$ given by 
$$
d_I(x)=e(I(R/xR))-e_{I(R/xR)}(0:x)
$$
   for  $x\in R$ such that $\dim (R/xR)=d-1$
can be expressed as
\begin{equation}\label{eq:Rees_Thm}
d_I(x)=\sum_{i=1}^rd_i(I)v_i(x),
\end{equation}
where $v_1,\ldots, v_r$ are the Rees valuations of $I$ and $d_i(I)\in \NN$ for all $i$. 
\end{theorem}

 The coefficients of $d_I(x)$ in  \autoref{ReesThm*}, \autoref{genReesThm*} and \autoref{mostgenReesThm*} are uniquely determined by \cite[Theorem 9.42]{Re2}. 

 \autoref{ReesThm*} is proven in  \autoref{AURD}, in  \autoref{CorThm2}.  \autoref{genReesThm*} and \autoref{mostgenReesThm*} are then easily deduced in  \autoref{SecLR} from  \autoref{genReesThm*}, using arguments of Rees.

 We now give the definition of an $\fm_R$-valuation, which can be found in \cite[\S 3]{ReSh}. 
Let $(R,\fm_R,\kk)$ be a $d$-dimensional local domain. An {\it $\fm_R$-valuation} of $R$ is a valuation $v$ of $R$ such that  its  valuation ring $(\mathcal O_v,\fm_v)$ dominates $R$ and $\mathcal O_v/\m_v$ is a finitely generated extension of $R/\m_R$ of transcendence degree $d-1$.    Equivalently,  $v$ is a divisorial valuation  with respect to $R$ that is centered at $\fm_R$, \cite[Definition 9.3.1]{HS}. 
More generally, if  $R$ is a $d$-dimensional local ring, then an $\fm_R$-valuation of $R$ is an $\fm_R$-valuation of $R/P$, where $P$ is a minimal prime of $R$ such that $\dim(R/P)=d$. 
If $R$ is a local ring, we define $\Div(\fm_R)$ to be the set of equivalence classes of $\fm_R$-valuations of $R$, where two $\fm_R$-valuations are equivalent if they have the same valuation ring.
We note that divisorial valuations have $\ZZ$ as their valuation group and are Noetherian by \cite[Corollary 6.3.5 and Theorem 9.3.2]{HS}.

The following proposition follows from  \cite[Theorem 9.41]{Re2}; we give a short proof of it at the end of \autoref{SecLR}.

\begin{proposition}\label{mval} In  \autoref{ReesThm*}, \autoref{genReesThm*}, and \autoref{mostgenReesThm*}, we have that 
$d_i(I)>0$ if and only if the Rees valuation $v_i$ is an $\fm_R$-valuation.
\end{proposition}

The next remark follows  from the proof of  \autoref{mval} or  from \cite[Proposition 10.4.3]{HS}.

\begin{remark} 
If $(R,\fm_R,\kk)$ is a formally equidimensional local ring and $I\subset R$ is an $\fm_R$-primary ideal, then all the Rees valuations $v_i$ of $I$ are $\m_R$-valuations. Therefore, if in \autoref{ReesThm*}, \autoref{genReesThm*}, or \autoref{mostgenReesThm*}, $R$ is  
formally equidimensional, then  $d_i(I)>0$ for all $i$.
\end{remark}

\begin{example}\label{ExampleC1}  \autoref{ReesLemma} is not always true for $x\in R$ such that $\dim(R/xR)=d-1$ if $x$ is not a nonzero divisor. 
\end{example}
\begin{proof}
Let $\kk$ be a field, $R=\kk[[x,y,z,w]]/(xyz,xyw)$, $f=w-z$. Then
$\dim(R)=3$. The prime ideals $P_1=(x)$ and $P_2=(y)$ have the property that $\dim(R/P_1)=\dim(R/P_2)=3$ and they are the only minimal primes of $R$ with this property. The other minimal prime is $P_3=(z,w)$ which satisfies $\dim R/(z,w)=2$. 
We have that $\dim(R/fR)=2$.
Take $I=\fm_R=(x,y,z,w)$. Notice $R/fR\cong \kk[[x,y,z]]/(xyz)$ so that $d_I(f)=e(I(R/fR))=3$.

We now calculate the expression
\begin{equation}\label{eqC1}
\ell_{R_{P_1}}(R_{P_1})d_{I(R/P_1)}(f)+\ell_{R_{P_2}}(R_{P_2})d_{I(R/P_2)}(f)
\end{equation}
of the right hand side of  \autoref{ReesLemma}. 
$R_{P_1}$ and $R_{P_2}$ are fields and $(R/P_1)/f(R/P_1)\cong \kk[[y,z,w]]/(w-z)\cong \kk[[y,z]]$ and
$(R/P_2)/f(R/P_2)\cong \kk[[x,z,w]]/(w-z)\cong \kk[[x,z]]$. Thus the expression in \autoref{eqC1} is equal to 2, which is not equal to $e(I(R/fR))$.
\end{proof}

\begin{example}\label{ExampleC2}
The conclusions of  \autoref{genReesThm*} are not always true for $x\in R$ such that $\dim(R/xR)=d-1$ if $x$ is not a nonzero divisor.
\end{example}
\begin{proof}
We will show that in the ring $R$ constructed in  \autoref{ExampleC1}, there do not exist functions $d_i(I)$ satisfying the conclusions of    \autoref{genReesThm*} for all $f\in R$ which satisfy $\dim(R/fR)=d-1$ (where $d=\dim R=3$). 

The conclusions of   \autoref{genReesThm*} are true for $f\in R$ which are nonzero divisors and such that $\dim(R/fR)=d-1$, so we do have an expression 
$$
d_I(f)=\sum_{i=1}^rd_i(I)v_i(f)
$$
for $f\in R$ such that $f$ is a nonzero divisor and $\dim(R/fR)=d-1$. Therefore, if such expression existed for all $f\in R$ such that $\dim R/fR=d-1$, it would have to be with this choice of $d_I(f)$.  The Rees valuations of $R$ are $v_1,v_2,v_3$ where $v_1$ is the $\fm$-adic valuation of $R/P_1\cong \kk[[y,z,w]]$, $v_2$ is the $\fm$-adic valuation of $R/P_2\cong \kk[[x,z,w]]$, and $v_3$ is the $\fm$-adic valuation of $R/P_3\cong \kk[[x,y]]$. Since \autoref{ReesLemma} is valid for nonzero divisors $f\in R$,  we have that $d_1(I)=1$, $d_2(I)=1$, and $d_3(I)=0$. Applying this formula to the element $f$ defined in the previous example, we obtain
$\sum_{i=1}^3d_i(I)v_i(f)=2$, which is not equal to $d_I(f)=3$.
\end{proof}

The ring $R$ constructed in  \autoref{ExampleC1} is an analytically unramified local ring, since it is complete and reduced. Thus  \autoref{ExampleC2} also gives a counterexample to \autoref{ReesThm*} when  $R$ is not a   domain, even if we assume that $\dim(R/xR)=d-1$.

 Finally, we verify that the element  $f\in R$ of \autoref{ExampleC1} and \autoref{ExampleC2} does satisfy the conclusions of \autoref{mostgenReesThm*}.

In the above examples, we have that
$$
(0:f)=xyR\cong R/(z,w)=R/P_3\cong \kk[[x,y]]
$$
and $I\kk[[x,y]]$ is the maximal ideal $\fm_{\kk[[x,y]]}$. Now $(0:f)$ is an $R/fR$-module, and computing the multiplicity 
$e_{I(R/fR)}(0:f)$ of $(0:f)$ with respect to $\fm_{R/fR}$ as an $R/fR$-module, we have that 
$e_{I(R/fR)}(0:f)=1$.  Therefore,  we obtain the correct conclusion
$$
e(I(R/fR))-e_{I(R/fR)}(0:f)=d_1(I)v_1(f)+d_2(I)v_2(f).
$$

\subsection{Analytically unramified local domains}\label{AURD}

In this subsection we prove a theorem about intersection theory over an analytically unramified local domain (\autoref{Theorem2}) and deduce several corollaries, including a geometric proof of \autoref{ReesThm*} in  \autoref{CorThm2}.

 We first interpret the theory of Rees valuations on an analytically unramified local domains geometrically in the following proposition.

\begin{proposition}\label{Reesval} Let 
$(R,\fm_R, \kk)$ be a $d$-dimensional analytically unramified local domain and 
 $I\subset R$ be an $\fm_R$-primary ideal. Let $\pi:Y=\Proj(\oplus_{n\gs 0}\overline{I^n})\rightarrow\Spec(R)$ be the normalization of the blowup of $I$. Then $Y$ is a projective $R$-scheme. Let $\Omega$ be the set of integral components of $\pi^{-1}(\fm_R)$. For $E\in\Omega$, the local ring $\mathcal O_{Y,E}$ is a discrete valuations ring. The   canonical valuations $v_E$ of $E\in \Omega$ are the Rees valuations of $I$.
 \end{proposition}

 We now prove \autoref{Reesval}.

 Let  $v_1,\ldots,v_r$ be the Rees valuations of $I$. 
There exist positive integers $e_1,\ldots,e_r$ such that  
 $$
 \overline{I^n}=I(v_1)_{ne_1}\cap \cdots \cap I(v_r)_{ne_r}
 $$
  for $n\gs 1$, where $I(v_i)_n=\{f\in R\mid v_i(f)\gs n\}$; see\cite[Ch. 10]{HS}. 
   Let $J(i)_n=\overline{I^n}\cap I(v_i)_{ne_i+1}$ for $n\gs 0$. 
   
 Let $t$ be an indeterminant. 
The graded $R$-algebra $B=\sum_{n\gs 0}\overline{I^n}t^n$ is finitely generated since $R$ is analytically unramified, see  \cite[Proposition 5.2.1 \& Corollary 9.2.1]{HS}. 
Then $Y=\Proj(B)$ is the normalization of the blowup of $I$, with natural projection $\pi:Y\rightarrow \Spec(R)$, so that $Y$ is a projective $R$-scheme.

Let $K$ be the quotient field of $R$. There exists an extension of $v_i$ to a valuation $\overline v_i$ of $K(t)$, defined as follows. We prescribe that $\overline{v}_i(t)=-e_i$, $\overline{v}_i(c)=v_i(c)$ if $0\ne c\in K$, and
$$
\overline{v}_i\Big(\sum_{0\neq c_n\in K} c_nt^n\Big)=\min_n\{v_i(c_n)-ne_i\}.
$$
 That this defines a valuation is classical; see \cite[Theorem 3.1]{SM}, where references are given for a proof of this result.

 Let $B^*=B[\frac{1}{t}]=\oplus_{n\in \ZZ}\overline{I^n}t^n$, where $\overline{I^n}=R$ if $n<0$. Let 
$$
\mathfrak p_i=\{x\in B^*\mid \overline v_i(x)\gs 1\},
$$
and note that $\mathfrak p_i$ is a homogeneous prime ideal of $B^*$, whose degree $n$ component is  
$$
(\mathfrak p_i)_n=\left\{\begin{array}{ll}
R&\mbox{ if }n<0\\
J(i)_n&\mbox{ if }n\gs 0.
\end{array}\right.
$$
Let $\fq_i=\{x\in B^*\mid \overline v_i(x)\gs e_i\}$ and notice that $\fq_i$ is a $\fp_i$-primary ideal.  
The equality $\frac{1}{t}B^*=\mathfrak q_1\cap \cdots\cap \mathfrak q_r$ implies $\frac{1}{t}B^*$ is an integrally closed ideal since it is an intersection of valuation ideals; see for example \cite[Lemma 4.6]{BDHM}. Therefore, by \cite[Lemma 4.2]{Re4} 
the localization $A_i=(B^*)_{\fp_i}$ is a regular local ring of dimension one, and hence is a valuation ring of dimension one.

 The valuation ring $\mathcal O_{{\overline v}_i}$ of $\overline \nu_i$ contains $A_i$ and  these rings have  the same quotient field. By \cite[Ch. VI, \S 9]{ZS2}, the rings associated to places $\mathcal P$ of a field $K$ are the valuation rings of $K$. Thus by \cite[Theorem VI.3.3]{ZS2}, 
$\mathcal O_{\overline\nu_i}$ is a localization of $A_i$ with respect to a prime ideal of $A_i$.
Since $A_i$ has dimension one and $\overline v_i$ is not the trivial valuation, we have that $A_i=\mathcal O_{{\overline v}_i}$.

Let $K^*=K(t)$ be the quotient field of $B^*$. Since $f\in K\cap A_i$ holds if and only if $f=\frac{t^ax}{t^ay}$ 
with $a\in \ZZ$ and $xt^a, yt^a\in [B^*]_a$  with $yt^a\not\in \mathfrak p_i$ (so that $a\gs 0$ as $\frac{1}{t}\in \mathfrak p_i$), we conclude that  $K\cap A_i=B_{(P_i)}$, where 
$P_i=B\cap \mathfrak p_i=\sum_{n\gs 0}J(i)_nt^n$ is a homogeneous primes ideal of $B$.
Now $0\ne f\in K\cap A_i$ if and only if $v_i(f)\gs 0$, since $\overline{v}_i$ is an extension of $v_i$ to $K(t)$. Thus $B_{(P_i)}$ is the valuation ring of $v_i$.
Let $E_i$ be the integral closed subscheme $E_i=\Proj(B/P_i)$ of $Y$. We have that
$\mathcal O_{Y,E_i}=B_{(P_i)}=\mathcal O_{v_i}$, so $E_i$ is the center of $v_i$ on $Y$.

\begin{lemma}\label{ExcSup} With the notation above, we have $\cap_{i=1}^rP_i=\sqrt{\fm_R B}$.
\end{lemma}

\begin{proof} We have that $\fm_R=I(v_i)_1$ for all $i$. If  $ft^n\in \fm_RB_n$, then $v_i(f)\gs ne_i+1$ for all $i$, which implies that
$ft^n\in \left( \cap_{i=1}^rP_i\right)_n$. Therefore,    $\fm_RB\subset \cap_{i=1}^rP_i$ and thus $\sqrt{\fm_RB}\subset \cap_{i=1}^rP_i$. 

Conversely, suppose that $ft^n\in \left(\cap_{i=1}^rP_i\right)_n$. We must find $m>0$ such that $f^mt^{mn}\in \fm_R\overline{I^{mn}}t^{mn}$. By \cite{Re5} or \cite[Theorem 9.1.2]{HS}, there exists $k>0$ such that $\overline{I^{n+k}}\subset I^n$ for all $n\gs 0$. Since  $v_i(f)\gs ne_i+1$ implies $v_i(f^m)\gs (mn+1+k)e_i$ for all $i$ if $m\gs (\max_ie_i)(k+1)$, for $m$ sufficiently large we have
$$
f^m\in \overline{I^{mn+1+k}}\subset I^{mn+1}\subset \fm_R\overline{I^{mn}},
$$
so that $\cap_{i=1}^rP_i\subset\sqrt{\fm_R B}$.
\end{proof}

\begin{remark} Since the support of $\Proj(B/\m_RB)$ is $\pi^{-1}(\fm_R)$, we have that $\pi^{-1}(\fm_R)=\cup_{i=1}^rE_i$ by  \autoref{ExcSup}.
\end{remark}

The above analysis proves  \autoref{Reesval}.

We now prove a theorem about intersection products from which we  deduce in \autoref{CorThm2} a geometric proof of  \autoref{ReesThm*}. 

\begin{theorem}\label{Theorem2} Let   $(R,\fm_R,\kk)$ be an  analytically unramified local domain  of dimension $d$. Let $I$ be an $\fm_R$-primary ideal of $R$ and  $\pi:Y=\Proj(\oplus_{n\gs 0}\overline{I^n})\rightarrow\Spec(R)$ be the natural projection.  Suppose that $F_1,\ldots, F_{d-1}$ are Cartier divisors on $Y$ with support above $\m_R$.  Then for any $0\ne x\in R$ we have
$$
-(F_1\cdot\ldots\cdot F_{d-1}\cdot T^*)_R=\sum_{E}v_E(x)(F_1\cdot\ldots
\cdot F_{d-1}\cdot E)_{R},
$$
where $T^*$ is the strict transform of $T=\Spec(R/xR)$ in $Y$ and 
$E$ ranges over the integral components of $\pi^{-1}(\fm_R)$. The local rings $\mathcal O_{Y,E}$ are discrete valuation rings and the canonical valuations $v_E$ of these $E$ are the Rees valuations of $I$. 
\end{theorem}

\begin{proof}  Let $E_1,\ldots,E_r$ be the integral components of $\pi^{-1}(\fm_R)$.  
The associated valuations $v_i=v_{E_i}$  are the Rees valuations of $I$ by  \autoref{Reesval}. 
Let $S=\Spec(R)$.  
Note  that $E_i\subset Y_{\kk}$ for all $i$. 

Using the notation in \autoref{SecCycle} and \autoref{SecRelInt}, we  have 
$$
[\mathcal O_Y/x\mathcal O_Y]=[\divv(x)]\in P_{-1}(Y/S),
$$
so that $[\mathcal O_Y/x\mathcal O_Y]=0$ in $A_{-1}(Y/S)$. Therefore,
$$
0=[\mathcal O_Y/x\mathcal O_Y]=[T^*]+\sum_{i=1}^r v_i(x)[E_i]
$$
in $A_{-1}(Y/S)$,
as  $\ell_{\mathcal O_{Y,E_i}}(\mathcal O_{Y,E_i}/x\mathcal O_{Y,E_i})=v_i(x)$ since $\mathcal O_{Y,E_i}$ is the valuation ring of $v_i$. Thus
$$
0=F_1\cdot\ldots\cdot F_{d-1}\cdot [T^*]+\sum_{i=1}^r v_i(x)F_1\cdot\ldots\cdot F_{d-1}\cdot [E_i]
$$
 in $A_{-d}(Y_{\kk}/S)=A_0(Y_{\kk}/\kk)$.

Therefore, computing the intersections products in $A_0(Y_{\kk}/\kk)$ over $\kk$,  we have that 
\begin{equation}\label{Req10}
0=\int_{Y_{\kk}}F_1\cdot\ldots\cdot F_{d-1}\cdot [T^*]+\sum_{i=1}^r v_i(x)\int_{Y_{\kk}}F_1\cdot\ldots\cdot F_{d-1}\cdot [E_i].
\end{equation}
Moreover, 
\begin{equation}\label{Req11}
\int_{Y_{\kk}}F_1\cdot\ldots\cdot F_{d-1}\cdot [E_i] 
=(F_1\cdot\ldots\cdot F_{d-1}\cdot E_i)_R
\end{equation}
and
\begin{equation}\label{Req12}
\int_{Y_{\kk}}F_1\cdot\ldots\cdot F_{d-1}\cdot [T^*] 
=(F_1\cdot\ldots\cdot F_{d-1}\cdot T^*)_R
\end{equation}
by \autoref{Proddef1}. 
The result now follows from  \autoref{Req10}, \autoref{Req11}, and \autoref{Req12}.
\end{proof}

\begin{corollary}\label{CorThm2}  Let   $(R,\fm_R,\kk)$ be an  analytically unramified local domain  of dimension $d$.  Let $I$ be an $\fm_R$-primary ideal of $R$ and $\pi:Y=\Proj(\oplus_{n\gs 0}\overline{I^n})\rightarrow\Spec(R)$ be the natural projection
 and let $\mathcal L_Y =I\mathcal O_Y$.  Then for any $0\ne x\in R$ we have
$$
e(I(R/xR))=\sum_{E} v_E(x)(\mathcal L_Y^{d-1}\cdot E)_{R},
$$
where $E$ ranges over the  integral components of $\pi^{-1}(\fm_R)$. 
The local rings $\mathcal O_{Y,E}$ are discrete valuation rings and the canonical valuations $v_E$ of these $E$ are the Rees valuations of $I$. 
\end{corollary}

\begin{proof} Let notation be as in the statement and proof of \autoref{Theorem2}.
We have that $T^*\rightarrow T$  is birational and projective and $I\mathcal O_{T^*}$ is invertible (by  \cite[Corollary II.7.15]{H}).
 Let $D$ be a Cartier divisor on $Y$ such that $I\mathcal O_Y=\mathcal O_Y(D)$, and let $F_i=D$ for $1\ls i\ls d-1$.
 Thus
 $$
 e(I(R/xR))=-((I\mathcal O_{T^*})^{d-1})_{R/xR}=
 -(D^{d-1}\cdot T^*)_R
 $$ 
 by  \autoref{TheoremMultInt} and the definition of the intersection product in  \autoref{Proddef1} and \autoref{Proddef2}. The result now   follows from \autoref{Theorem2}.
 
\end{proof}

 It follows from \autoref{CorThm2}  that the $d_i(I)$ in  \autoref{ReesThm*} can be calculated by an intersection product.

\begin{corollary}\label{Theorem2*} Let $(R,\fm_R,\kk)$ be a $d$-dimensional  analytically unramified local domain and $I\subset R$ be an $\fm_R$-primary ideal.  Let $\pi:Y=\Proj(\oplus_{n\gs 0}\overline{I^n})\rightarrow\Spec(R)$ be the natural projection. Index the integral components of $\pi^{-1}(\fm_R)$ as $E_1,\ldots,E_r$, where the valuation ring of $v_i$ is $\mathcal O_{Y,E_i}$ (as shown above). Let $d_i(I)$ be the coefficients of the expansion 
$
d_I(x)=\sum_{i=1}^rd_i(I)v_i(x)
$ 
in  \autoref{ReesThm*}. Then 
$$
d_i(I)=((I\mathcal O_Y)^{d-1}\cdot E_i)_R
$$
for $1\ls i\ls r$.
\end{corollary}

The conclusions of \autoref{Theorem2} and \autoref{CorThm2} are achieved for more general birational proper morphisms, as the following corollaries show.

\begin{corollary}\label{Cor2}
Let $(R,\fm_R,\kk)$ be a  local domain. Further assume  that $\phi:W\rightarrow \Spec(R)$ is a proper birational morphism such that 
$W$ is normal. Suppose that $F_1,\ldots, F_{d-1}$ are Cartier divisors on $W$ with support above $\m_R$.  Then for any $0\ne x\in R$ we have
$$
-(F_1\cdot\ldots\cdot F_{d-1}\cdot T^*)_R=\sum_{E}v_E(x)(F_1\cdot\ldots\cdot F_{d-1}\cdot E)_{R},
 $$
where $T^*$ is the strict transform of $T=\Spec(R/xR)$ in $W$ and 
$E$ ranges over the prime  divisors $E$ which are supported above $\fm_R$.
\end{corollary}

\begin{proof} 
All local rings $\mathcal O_{W,E}$ are valuation rings since $W$ is normal. Therefore, with
  small modifications,  the proof of \autoref{Theorem2} extends to the  situation in this result. 
\end{proof}

\begin{corollary}\label{Cor2*}
Let $(R,\fm_R,\kk)$ be a   $d$-dimensional local domain and $I\subset R$ be an $\fm_R$-primary ideal. Further assume  that $\phi:W\rightarrow \Spec(R)$ is a proper birational morphism such that 
 $W$ is normal, and $\L_W= I\mathcal O_W$ is invertible. Then  for any $0\ne x\in R$,
$$
e(I(R/xR))=\sum_{v\in \Div(\fm_R)}v(x)(\L_W^{d-1}\cdot C(W,v))_{R},
 $$
where the sum is over the  $\fm_R$-valuations   $v\in \Div(\fm_R)$ of $R$.  Here $C(W,v)$ is the center of $v$ on $W$, i.e.,  the integral subscheme  of $W$ that is the closure of the unique point $q$ of $W$ such that the local ring $\mathcal O_{Y,q}$ is dominated by the valuation ring  of $v$.
\end{corollary}

\begin{proof} Let $T^*$ be the strict transform of $T=\Spec(R/xR)$ in $W$. The induced morphism $T^*\rightarrow T$ is proper and birational and $I\mathcal O_{T^*}$ is invertible. Thus
$$
e(I(R/xR))=-((I\mathcal O_{T^*})^{d-1})_{R/xR}=-((I\mathcal O_W)^{d-1}\cdot T^*)_R
$$
by \autoref{TheoremMultInt}.  
By \autoref{Cor2}, taking $F_i=D$ for $1\ls i\ls d-1$,  where $D$ is a Cartier divisor on $W$ such that $I\mathcal O_W= \mathcal O_W(D)$,
we have that 

$$
-(\mathcal L_W^{d-1}\cdot T^*)_R=\sum_E v_E(x)(\mathcal L_W^{d-1}\cdot E)_R,
$$
where the sum is over the prime  divisors $E$ of $W$ which are supported above $\fm_R$.

%If $v\in \Div(\fm_R)$ and $\dim C(W,v)=d-1$, then $C(W,v)$ is a prime divisor $E$ of $W$ which is supported above $\fm_R$.
% If $\dim C(W,v)<d-1$, then $(\mathcal L_W^{d-1}\cdot C(W,v))_R=0$ by \autoref{Proddef1}. The conclusions %of the corollary now follow.

We must show that the canonical valuations $v_E$ associated to integral components $E$ of $\phi^{-1}(\fm_R)$ such that $\dim(E)=d-1$ are a complete set of representatives of the 
equivalence classes of valuations in $\Div(\fm_R)$ such that $\dim(C(W,v))=d-1$.

We first show that if $v\in \Div(\fm_R)$ is such that $\dim(C(W,v))=d-1$, then $v$ is equivalent to $v_E$ for some integral component $E$ of $\pi^{-1}(\fm_R)$ such that $\dim(E)=d-1$.

Suppose that $v\in \Div(\fm_R)$ and $\dim\left(C(W,v)\right)=d-1$. Then $E=C(W,v)$ is an integral component of $\pi^{-1}(\fm_R)$ with $\dim(E)=d-1$ and $\mathcal O_v$ dominates $\mathcal O_{W,E}$. Since $\mathcal O_{W,E}$ and $\mathcal O_v$ are dimension 1 valuation rings of the quotient field of $R$, we have that $\mathcal O_v=\mathcal O_{W,E}$ by \cite[Theorem VI.3.3]{ZS2}, as in the proof of \autoref{Reesval}. Thus $v$ is equivalent to $v_E$.

Finally, if $E$ is an integral component of $\phi^{-1}(\fm_R)$ such that $\dim(E)=d-1$, then $E$ is a projective $\kk$-variety such that $\trdeg(R(E)/\kk)=\dim(E)=d-1$. Thus $v_E\in \Div(\fm_R)$.  If $E_1$ and $E_2$ are distinct integral subschemes of $\phi^{-1}(\fm_R)$ then $\mathcal O_{W,E_1}\ne \mathcal O_{W,E_2}$ since $\phi$ is proper. Thus if $E_1$ and $E_2$ have dimension $d-1$ we have that $v_{E_1}$ is not equivalent to $v_{E_2}$, finishing the proof.
\end{proof}

\subsection{Local rings}\label{SecLR}
In this subsection we give  proofs of \autoref{genReesThm*}  and \autoref{mostgenReesThm*} and give a geometric interpretation in  \autoref{ThmExc}.

We  require the following two  lemmas giving basic properties of Rees valuations to obtain the full statement of  \autoref{genReesThm*}. If $v$ is a discrete valuation on a local ring $R$ and $n\in \NN$, then define the valuation ideal $I(v)_n=\{x\in R\mid v(x)\gs n\}$.

\begin{lemma}\label{LemmaR1} Let $(R,\fm_R,\kk)$ be a local ring and $I$ be an $\fm_R$-primary ideal. Let $P_1,\ldots,P_r$ be the minimal prime ideals of $R$. Then the set of Rees valuations of $I$ is the union of the sets of Rees valuations of $I(R/P_i)$ for $1\ls i\ls r$.
\end{lemma}

\begin{proof} Let $\phi_i:R\rightarrow R/P_i$ be the natural surjections. By \cite[Proposition 1.1.5]{HS}, 
$$\overline{I^n}=\cap_{i=1}^t\phi_i^{-1}(\overline{I^n(R/P_i)})
$$
for all $n$. Let $v_{i,j}$ for $1\ls j\ls r_i$ be the Rees valuations of $I(R/P_i)$, giving irredundant representations
$$
\overline{I^n(R/P_i)} = \cap_{i=1}^{r_i}I(v_{ij})_{ne_{i,j}}.
$$
Identify $v_{i,j}$ with $v_{i,j}\circ \phi_i$ (so that $v_{i,j}(P_i)=\infty)$. Then
$\overline{I^n}=\cap_{i,j}I(v_{i,j})_{ne_{i,j}}$, so we only need to show that this representation is irredundant. To prove this, we use an argument from the proof of \cite[Theorem 4.12]{Re2}.

Fix $i,k$ with $1\ls i\ls r$ and $1\ls k\ls r_i$. To prove irredundancy, we must find $x\in \overline{I^m}$ for some $m$, such that
\begin{equation}\label{eqR1}
\frac{v_{i,k}(x)}{me_{i,k}}<\frac{v_{j,l}(x)}{m_{e_{j,l}}}
\end{equation}
for all $(j,l)\ne (i,k)$.
By definition of Rees valuations, there exists $y\in \overline{I^m}$ such that
\begin{equation}\label{eqR2}
\frac{v_{i,k}(y)}{me_{i,k}}<\frac{v_{i,l}(y)}{me_{i,l}}
\end{equation}
for $l\ne k$. Choose $z\in \cap_{j\ne i}P_j\setminus P_i$. Then for $N>0$, $y^Nz\in \overline{I^m}$.
We compute $v_{j,l}(y^Nz)=\infty$ for $j\ne i$ and
$$
\frac{v_{i,l}(y^Nz)}{me_{i,l}}=N\frac{v_{i,l}(y)}{me_{i,l}}+\frac{v_{i,l}(z)}{me_{i,l}}.
$$
By \autoref{eqR2}, for $N\gg 0$, 
$$
\frac{v_{i,k}(y^Nz)}{me_{i,k}}<\frac{v_{i,l}(y^Nz)}{me_{i,l}}
$$
for $l\ne k$. Taking $x=y^Nz$ with $N\gg 0$, we have that \autoref{eqR1} holds. 
\end{proof}

The following lemma follows from the more general \cite[Theorem 5.3]{KV}. We give a short proof here.

\begin{lemma}\label{LemmaR2}
Let $(R,\fm_R,\kk)$ be a local ring and $I$ be an $\fm_R$-primary ideal. Let $\hat R$ be the $\fm_R$-adic completion of $R$ and $\hat v_i$ for $1\ls i\ls t$ be the Rees valuations of $I\hat R$.  Then $v_i=\hat v_i|R$ are the Rees valuations of $I$.  Further, the value groups and the residue fields of $v_i$ and $\hat v_i$ are naturally isomorphic. 
\end{lemma}

\begin{proof} $\overline{I^n\hat R}=\overline{I^n}\hat R$ by \cite[Lemma 9.1.1]{HS}. By assumption, there exists an irredundant representation
$$
\overline{I^n}\hat R=\cap_{i=1}^tI(\hat v_i)_{ne_i}
$$
for $n\gs 0$. Since $R\rightarrow \hat R$ is faithfully flat, $(\overline I^n\hat R)\cap R=\overline{I^n}$ for all $n$. Thus
$$
\overline{I^n}=\left[\cap_i(I(\hat v_i)_{ne_i})\right]\cap R=\cap_i\left[I(\hat v_i)_{ne_i}\cap R\right]
=\cap_iI(v_i)_{ne_i}.
$$
We now prove irredundancy. Fix $i$ with $1\ls i\ls t$. By assumption, there exists $y\in \overline{I^m}\hat R$ for some $m$ such that 
$$
\frac{\hat v_i(y)}{me_i}<\frac{\hat v_j(y)}{me_j}
$$
if $j\ne i$. There exists $N>0$ such that $N\hat v_j(\fm_{\hat R})>\hat v_j(y)$ for $1\ls j\ls t$, and there exists $z\in R$ such that $z-y\in \fm_{\hat R}^N$. Thus $v_j(z)=\hat v_j(z)=\hat v_j(y)$ for all $j$, which implies that
$$
\frac{v_i(z)}{me_i}<\frac{v_j(z)}{me_j}
$$
for $j\ne i$.
The statement about equality of value groups and residue fields follows from \cite[Proposition 9.3.5]{HS}.
\end{proof}

We now give a proof of  \autoref{genReesThm*}. We proceed as in \cite[\S 3]{Re}. Let $P_i$ be the minimal primes of $\hat R$, indexed so that $\dim( \hat R/P_i)=d$ for $1\ls i\ls s$ and $\dim(\hat R/P_i)<d$ if $s<i$. Let $\hat v_{i,j}$ be the Rees valuations of $I(\hat R/P_i)$.
We have that $e(I(R/xR))=e(I(\hat R/x\hat R))$, where $\hat R$ is the $\fm_R$-adic completion of $R$. By  \autoref{ReesLemma},
\begin{equation}\label{eqR3}
e(I(\hat R/x\hat R))=\sum_{1\ls i\ls s} \ell_{\hat R_{P_i}}(\hat R/P_i)d_{I(\hat R/P_i)}(x),
\end{equation}
where the $P_i$ are the minimal prime ideals of $\hat R$ such that $\dim(\hat R/P_i)=d$. The ring $\hat R/P_i$ is analytically irreducible since it is a complete domain, so we have expressions for $1\ls i\ls s$,
\begin{equation}\label{eqR4}
d_{I(\hat R/P_i)}(x)=\sum_jd_{ij}\hat v_{i,j}(x)
\end{equation}
by  \autoref{ReesThm*} or by  \autoref{CorThm2}. Set $d_{i,j}=0$ for $i>s$.
Combining \autoref{eqR3} and \autoref{eqR4}, we obtain an expansion of 
$d_I(x)$ in terms of the $\hat{v}_{i,j}$.
By  \autoref{LemmaR1}, the union of all the  $\hat v_{i,j}$ is the set of Rees valuations of $I\hat R$,  and by  \autoref{LemmaR2}, the $\hat v_{i,j}$ are the extensions of the Rees valuations of $I$ to $\hat R$.

We have the following geometric interpretation of  \autoref{genReesThm*}.  Excellent local rings satisfy the assumptions of the theorem by  \cite[Scholie IV.7.8.3]{EGAIV}.  

\begin{theorem}\label{ThmExc} Let  $(R,\fm_R,\kk)$ be a $d$-dimensional  local ring such that $R/P$ is analytically unramified  for every  prime ideal $P$ such that $\dim(R/P)=d$. Let $I$ be is an $\fm_R$-primary ideal such that $\oplus_{n\gs 0}\overline {I^n}$ is a finitely generated $R$-algebra.
Let $\pi:\Proj(\oplus_{n\gs 0}\overline{I^n})\rightarrow \Spec(R)$ be the natural projective morphism and set $\mathcal L_Y=I\mathcal O_Y$.
Then the degree function $d_I: R\setminus\{0\}\to\NN$ given by $d_I(x)=e(I(R/xR))$   for  $x\in R$ such that $\dim (R/xR)=d-1$
 and $x$ is a nonzero divisor on $R$ can be expressed as
$$
d_I(x)=\sum_va_v(\mathcal L_Y^{d-1}\cdot C(Y,v))_Rv(x),
$$
where the sum is over the Rees valuations v of $I$. Here $C(Y,v)$ is the center of $v$ on $Y$, i.e.,  the integral subscheme  of $Y$ that is the closure of the unique point $q$ of $Y$ such that the local ring $\mathcal O_{Y,q}$ is dominated by the valuation ring  of $v$ and   $a_v=\ell_{R_{P_v}}(R_{P_v})$
with $P_v=\{x\in R\mid v(x)=\infty\}$.
\end{theorem}
\begin{proof} Let $P_1,\ldots,P_s$ be the minimal prime ideals of $R$ such that $\dim(R/P_i)=d$. 
We have that $\overline{I^n}(R/P_i)=\overline{I^n(R/P_i)}$ for all $n$ by \cite[Proposition 1.1.5]{HS}. Thus the strict transform $Y_i$ of $\Spec(R/P_i)$ in $Y$ is $Y_i=\Proj(\oplus_{n\gs 0}\overline{I^n(R/P_i))}$. We have  natural projective birational morphisms $\pi_i:Y_i\rightarrow \Spec(R/P_i)$, and closed immersions $\lambda_i:Y_i\rightarrow Y$.  Then $\mathcal L_{Y_i}:=\lambda_i^*\mathcal L_Y\cong I\mathcal O_{Y_i}$ for every $i$ and, for a Rees valuation $v$ of $I(R/P_i)$ we have  $(\lambda_i)_*C(Y_i,v)=C(Y,v)$. Thus
\begin{equation}\label{eqR6}
(\mathcal L_{Y_i}^{d-1}\cdot C(Y_i,v))_{R/P_i}=(\mathcal L_Y^{d-1}\cdot C(Y,v))_R
\end{equation}
by part (c) of \autoref{Prop2.3} and \autoref{Proddef1}.
The theorem now follows from \autoref{eqR6},  \autoref{ReesLemma},  \autoref{CorThm2}, and   \autoref{LemmaR1}.
\end{proof}

We now give a proof of  \autoref{mostgenReesThm*}, as a consequence of \autoref{ReesThm*}.  The original proof in \cite{Re2} is part of a general theory developed in the course of nine chapters. Since it is somewhat difficult  to extract the key elements of this proof  from \cite{Re2}, we include below a condensed version of it  for the reader's convenience.

If $(A,\fm_A)$ is a local ring, $J$ is an $\fm_A$-primary ideal, and $N$ is a finitely generated $A$-module, then
$e_J(N)$ will denote the multiplicity of $N$ with respect to $J$. 

We extend the definition of the degree function $d_I(x)$ to finitely generated modules $M$ over $R$ by defining
$$
d_I(M,x)=e_{I(R/xR)}(M/xM)-e_{I(R/xR)}(0:_Mx)
$$
for $x\in R$ such that $\dim(R/xR)=d-1$.

In \cite[Lemma 9.11]{Re2}, it is shown that for fixed $x$ such that $\dim(R/xR)=d-1$, $d_I(M,x)$ is an additive functor on the category of finitely generated $R$-modules, which assumes only nonnegative integers as values. Furthermore, if $P$ is a prime ideal of $R$ such that $\dim(R/P)<d$, then $d_I(R/P,x)=0$.

Let $M=\hat R$ be the $\fm_R$-adic completion of $R$, and let 
$$
M=M_0\supset M_1\supset \cdots\supset M_n=(0)
$$
be a filtration of $M$ in which $M_{i-1}/M_i\cong \hat R/P_i$ for some prime ideal $P_i$ of $\hat R$ for all $i$. For a minimal prime $P$ of $\hat R$, let $\tau(P)$ be the number of $i$ such that $P_i=P$. Using the facts from \cite[Lemma 9.11]{Re2} stated above, we have that
$$
d_I(x)=d_{I\hat R}(x)=\sum_P \tau(P)d_{I(\hat R/P)}(x),
$$
where the sum is over the minimal primes $P$ of $\hat R$ such that $\dim(\hat R/P)=d$. Now apply  \autoref{ReesThm*} to compute $d_{I(\hat R/P)}(x)$ if $\dim (\hat R/P)=d$ and then use  \autoref{LemmaR2} and \autoref{LemmaR1}, defining $d_i(I)=0$ if $\dim (\hat R/P_i)<d$, to obtain the conclusions of  \autoref{mostgenReesThm*}.

Now we give the proof of  \autoref{mval}. We first prove the proposition in the case that $R$ is an analytically unramified local domain.  The following facts are proven in \autoref{Reesval} and
\autoref{CorThm2}. 
The morphism $\pi:Y=\Proj(\oplus_{n\gs 0}\overline{I^n})\rightarrow 
\Spec(R)$ is a birational projective morphism. The Rees valuations $v_i$ of $I$ are in one-to-one correspondence with the integral components $E_i$ of $\pi^{-1}(\fm_R)$. Each $\mathcal O_{Y,E_i}$ is a discrete valuation ring and $v_i$ is equivalent to the canonical valuation $v_{E_i}$ of $\mathcal O_{Y,E_i}$.
Further, $d_i(I)=(\mathcal L_Y^{d-1}\cdot E_i)_R$, where $\mathcal L_Y=I\mathcal O_Y$.

For each $i$, we have that $\dim(E_i)\ls d-1$ and $\dim(E_i)=d-1$ if and only if $v_i$ is an $\fm_R$-valuation, since $\dim(E_i)= \trdeg(R(E_i)/\kk)$  as $E_i$ is a projective $\kk$-variety. We have that
$(\mathcal L_Y^{d-1}\cdot E_i)_R=0$ if $\dim(E_i)<d-1$. Now $\mathcal L_Y$ is ample on $Y$  (by \cite[Exercise III.5.7 (d)]{H}), since it is the pullback of the ample divisor $I\mathcal O_X$ on $X=\Proj(\oplus_{n\gs 0}I^n)$ and $Y\rightarrow X$ is a finite surjective morphism. Thus $\mathcal L_Y\otimes\mathcal O_{E_i}$ is ample on $E_i$. We have that  $(\mathcal L_Y^{d-1}\cdot E_i)_R=((\mathcal L_Y\otimes E_i)^{d-1})_{\kk}>0$ if $\dim(E_i)=d-1$, see for instance  \cite[Theorem III.1.1]{Kl}. This completes the proof of \autoref{mval} when $R$ is an analytically unramified local domain.  

To see that the conclusion of the proposition holds under the assumptions of  \autoref{genReesThm*} or  \autoref{mostgenReesThm*}, we note that by \autoref{LemmaR1} and \autoref{LemmaR2},  the Rees valuations $\hat v_{i,j}$ of $I{\hat R/P_i}$, where $P_i$ are the minimal primes of $\hat R$, restrict to the Rees valuations $v_{i,j}$ of $I$. 
Thus the conclusion follows from the first part of the proof.

\section{Mixed multiplicities of ideals as intersection products} \label{SecMix}

Let $(R,\fm_R,\kk)$ be a $d$-dimensional local ring  and   $I_1,\ldots, I_r$ be $\fm_R$-primary ideals of $R$. It follows from the theory of multigraded Hilbert polynomials that there exists a polynomial 
$P_{I_\bullet}(t_1,\ldots, t_r)\in \QQ[t_1,\ldots, t_r]$ of degree $d$ such  that  
\begin{equation}\label{eq:mult_Hilb_Pol}
P_{I_\bullet}(n_1,\ldots, n_r)=\ell_R(R/I_1^{n_1}\ldots I_r^{n_r})
\quad \text{for integers} \quad
n_1,\ldots, n_r\gg 0,
\end{equation}
see \cite[\S17.4]{HS},  \cite{Bh}, or \cite[\S1.2]{T}. Moreover, the degree $d$ homogeneous term of $P_{I_\bullet}(t_1,\ldots, t_r)$ can be written as 
\begin{equation}\label{eq:mixed_mult}
\sum_{v_1+\cdots+v_r=d}\frac{e(I_1^{[v_1]},\ldots, I_r^{[v_r]})}{v_1!\cdots v_r!}t_1^{v_1}\cdots t_r^{v_r},
\end{equation}
with nonnegative integers $e(I_1^{[v_1]},\ldots, I_r^{[v_r]})$ called the {\it mixed multiplicities} of $I_1,\ldots, I_r$. If $v_i=1$ we denote $I_i^{[1]}$ by $I_i$.

The following well-known proposition offers a convenient way of expressing mixed multiplicities. Here, for $v_1+\cdots +v_r=d$ we use the standard notation $\binom{d}{v_1,\ldots, v_r}=\frac{d!}{v_1!\cdots v_r!}$.

\begin{proposition}\label{prop:mixed_as_HS}
Let $(R,\fm_R,\kk)$ be a $d$-dimensional local ring  and   $I_1,\ldots, I_r$ be $\fm_R$-primary ideals of $R$.
For any $n_1,\ldots, n_r\in \ZZ_{>0}$ we have
$$
e(I_1^{n_1}\cdots I_r^{n_r}) = \sum_{\stackrel{v_1,\ldots, v_r\in \NN,}{v_1+\cdots+v_r=d}}\binom{d}{v_1,\ldots, v_r}e(I_1^{[v_1]},\ldots, I_r^{[v_r]})n_1^{v_1}\cdots n_r^{v_r}.
$$
\end{proposition}
\begin{proof}
The conclusion follows from \autoref{eq:mult_Hilb_Pol}, \autoref{eq:mixed_mult}, and the definition of Hilbert-Samuel multiplicity 
$$
e(I_1^{n_1}\cdots I_r^{n_r})=
\lim_{n\to \infty}\frac{\ell_R(R/I_1^{n n_1}\ldots I_r^{n n_r})d!}{n^d}.
$$
\end{proof}

In the following theorem we provide a formulation of the mixed multiplicities of ideals in terms of intersection theory.

\begin{theorem}\label{thm:MixedMult}
Let $(R,\fm_R,\kk)$ be a $d$-dimensional local ring  and   $I_1,\ldots, I_r$ be $\fm_R$-primary ideals of $R$. Let  $\pi:Y\rightarrow \Spec(R)$ be a proper birational morphism such that $I_i\mathcal O_Y$ is locally principal for every $i$. 
 Then for every $v_1,\ldots, v_r\in \NN$ such that $v_1+\cdots+v_r=d$ we have
$$ 
e(I_1^{[v_1]},\ldots, I_r^{[v_r]})=-\big((I_1\mathcal O_Y)^{v_1}\cdot\ldots\cdot(I_r\mathcal O_Y)^{v_r}\big)_R.
$$
\end{theorem}
\begin{proof} We have for all $n_1,\ldots,n_r\in\NN$ that
$$
e(I_1^{n_1}\cdots I_r^{n_r})=-((I_1^{n_1}\cdots I_r^{n_r}\mathcal O_Y)^d)_R
=
-\sum_{\stackrel{v_1,\ldots, v_r\in \NN,}{v_1+\cdots+v_r=d}}\binom{d}{v_1,\ldots, v_r}
((I_1\mathcal O_Y)^{v_1}\cdot\ldots\cdot (I_r\mathcal O_Y)^{v_d})_R
n_1^{v_1}\cdots n_r^{v_r}
$$
By \autoref{TheoremMultInt} and  \autoref{CorR18}. We obtain the conclusions of the theorem by comparing coefficients in the above polynomial and the polynomial of  \autoref{prop:mixed_as_HS}.

\end{proof}

\begin{remark}The conclusion of \autoref{thm:MixedMult}  is stated in \cite[Section 1.6B]{Laz1} under the assumption that  $R$ is the local ring of a closed point on a finite type scheme over an algebraically closed field. Related formulas  appear in \cite{KT}. 
\end{remark}

As a consequence of \autoref{thm:MixedMult} and  \autoref{CorR18}, we obtain the following identity for mixed multiplicities of ideals in  arbitrary local rings. 
Under more assumptions on $R$, a related formula appeared in  \cite[Theorem 8.1]{KK}.

\begin{corollary}\label{cor:Ident_Mixed_M}
Let $(R,\fm_R,\kk)$ be a $d$-dimensional local ring,   and let  $I$ and $J$ be $\fm_R$-primary ideals of $R$. For every $v_1,\ldots, v_r\in \NN$ such that $v_1+\cdots+v_r=d$ we have
$$ 
e\big((IJ)^{[v_1]},I_2^{[v_2]}\ldots, I_r^{[v_r]}\big)=
\sum_{i=0}^{v_1}\binom{v_i}{i}
e\big(I^{[i]},J^{[v_1-i]},I_2^{[v_2]},\ldots, I_r^{[v_r]}\big).
$$
In particular, if $v_1=1$ 
then 
$
e\big(IJ,I_2^{[v_2]}\ldots, I_r^{[v_r]}\big)=
e\big(I,I_2^{[v_2]},\ldots, I_r^{[v_r]}\big)
+
e\big(J,I_2^{[v_2]},\ldots, I_r^{[v_r]}\big).
$
\end{corollary}

\subsection{The semigroup and group of integrally closed ideals on a $2$-dimensional local ring}\label{SecSemi}
Let $(R,m_R,\kk)$ be a $2$-dimensional  local ring. The semigroup and group of integrally closed ideals of $R$ are defined on  \cite[p.461]{ReSh}. The multiplicative semigroup of $\fm_R$-primary ideals of $R$ has the congruence relation $\sim$ defined by $I\sim J$ if $I$ and $J$ have the same integral closure. Let $S(R)$ be the semigroup which is the quotient of the multiplicative semigroup by $\sim$. 

For $R$  normal, we say  that an $\m_R$-primary ideal $I$ is {\it $*$-simple} if the integral closure $\overline{I}$ of $I$ satisfies  that $\overline I\ne \overline{JK}$ for any two proper ideals $J$ and $K$ of $R$. This gives a notion of an irreducible element in $S(R)$. We note that
 mild assumptions on the structure of  $S(R)$ translate to mild  singularities of $R$. 
For example,    Zariski proved that if $R$ is regular, then $S(R)$ has unique factorization \cite[Appendix 5, Theorem 3]{ZS2} (a more recent proof is given by Huneke in \cite{Hu}). Moreover, in \cite[Corollary to Theorem 5]{C5} the first author shows that if  $R$ is complete  with algebraically closed residue field, then $S(R)$ has unique factorization if and only if $R$ is a UFD, where necessity follows from \cite[Theorem 5]{C5} and sufficiency  from 
\cite[Theorem 20]{L}.

Since  $S(R)$ satisfies the cancellation law \cite[Exercise 1.2]{HS}, it can be embedded in an Abelian group $G(R)$, which is the group defined by taking fractions of $S(R)$.  We note that the Cohen-Macaulay assumption in \cite[p.461]{ReSh} is not necessary for this conclusion.   
We  denote the class of the ideal $I$ in $S(R)$ by $\overline{I}$, the integral closure of $I$. Rees and Sharp define 
the symmetric bilinear function
$e^*$ on $G(R)\times G(R)$ by 
$$
e^*(\overline{I_1}\,\overline{I_2}^{-1}
\mid
\overline{J_1}\,\overline{J_2}^{-1})
=e(I_1,J_1)-e(I_1,J_2)-e(I_2,J_1)+e(I_2,J_2)
$$
and consider 
the associated quadratic form $q^*$ on $G(R)$, defined by
$$
q^*(x)=e^*(x,x).
$$
In  \cite[Theorem 5.4]{ReSh} they show  $q^*$ is positive definite with the assumption that $R$ is a Cohen-Macaulay local ring. 
If  $R$ is is a domain, 
this theorem  
is in fact a consequence of the negative definiteness of the intersection matrix of a resolution of singularities, as we now show.

Let $\pi:Y\rightarrow \Spec(R)$ be a birational proper morphism.  On   \cite[p. 223]{L} an intersection product 
$$
(F\cdot C)=\chi_{\kk}(\mathcal O_Y(F)\otimes\mathcal O_C)-\chi_{\kk}(\mathcal O_C)
$$
is defined for a Cartier divisor $F$ on $Y$, which we assume has  support above $\fm_R$, and 
a curve $C$ on $Y$ with  support above $\fm_R$. A curve is defined to be a closed one dimensional subscheme of $\pi^{-1}(\fm_R)$. By the definition, 
$$
(F\cdot C)=\chi_{\kk}(\mathcal O_Y(F)\otimes\mathcal O_C)-\chi_{\kk}(\mathcal O_C)
=I(\mathcal O_Y(F)\otimes\mathcal O_C)_{\kk}=(F\cdot C)_R,
$$
where the second equality is by \cite[Definition B.8]{Kl2} and the definition of $c_1(\mathcal L)\mathcal O_C$ before \cite[Definition B.2]{Kl2}. Thus the intersection product used in \cite{L} is the same as the one used in this paper in the case of two dimensional local rings.

 A resolution of singularities of a scheme $X$ is a scheme $Y$ with a proper birational morphism $Y\rightarrow X$ such that $\mathcal O_{Y,p}$ is a regular local ring for all $p\in Y$. 
The existence of resolution of singularities of two dimensional reduced excellent schemes is proven (with very different proofs) in \cite{L2} and \cite{CJS}. An exposition of \cite{L2} is given in \cite{Artin}.

The following theorem is a statement in our language of the fact that the intersection matrix of a resolution of singularities of a surface singularity is negative definite \cite{Mu2}, \cite{L}.

\begin{theorem}\label{CorNegDef} 
Let $(R,\fm_R,\kk)$ be a $2$-dimensional excellent  local domain. Let $\pi:Y\rightarrow \Spec(R)$ be a birational proper morphism with $Y$  nonsingular. Let $E_1,\ldots,E_r$ be the integral curves of $Y$ which lie above $\fm_R$.  Then the intersection matrix $M=\left((E_i\cdot E_j)_R\right)$ is negative definite. 
In particular, the quadratic form 
$Q(D)=(D\cdot D)_R$ on the free abelian group of exceptional divisors of $\pi$ is negative definite.
\end{theorem}

\begin{proof} 
 Let $S=\overline{R}$ be the normalization of $R$. Let $\m_i$ be the maximal ideals of $S$ and let  $S_i=S_{\m_i}$ and $Y_i=Y\times_{\Spec(S)}\Spec(S_i)$.
 Let $D$ be a  nonzero divisor  on $X$ with support above $\fm_R$.    
We must show that $(D^2)<0$.

Write  $D=\sum_{i=1}^rD_i$, where  $D_i$ is supported above $\fm_i$ for each $i$, and so can be considered to be a Cartier divisor on $Y_i$. 
We have that  
$$
(D^2)_R=\sum_{i=1}^r[S/\m_{i}:\kk](D_i^2)_{ S_i}
$$
by  \autoref{Theoremnormal*}.
By \cite[Lemma 14.1]{L},  for each $i$, the intersection matrix $\left((E_a\cdot E_b)_{S_i}\right)$ is negative definite; here  the intersection matrix is over the $E_a$ and $E_b$ that contract to $\fm_i$.  Since $(E_a\cdot E_b)=0$ if $E_a$ and $E_b$ contract to different points of $S$, we have that 
$(D^2)_R<0$.
\end{proof}

We derive two theorems in commutative algebra from \autoref{CorNegDef}.
The following  inequality is proved  for 2-dimensional Cohen-Macaulay analytic local rings in \cite{ELT} and for arbitrary two dimensional local rings in \cite{ReSh} and \cite[Lemma 17.7.1]{HS}.    

Our proof is similar to the one in \cite{ELT} but our proof  is true a little more generally. The proof in \cite{ELT}  reduces to an application of the theorem that the intersection matrix of a resolution of singularities is negative definite, which is the method we  also use  in our proof. 

\begin{corollary}\label{ineq2} Let $(R,\fm_R,\kk)$ be a 2-dimensional excellent local ring,   and let  $I$ and $J$ be $\fm_R$-primary ideals of $R$. Then
$$
e(I,J)^2\ls e(I)e(J).
$$
\end{corollary}

\begin{proof} Let $\pi:Y\rightarrow \Spec(R)$ be a birational projective morphism such that $I\mathcal O_Y$ and $J\mathcal O_Y$ are invertible. Let $\{P_i\}$ be the minimal primes of $R$ such that $\dim(R/P_i)=2$ and $Y_i$ be the integral component of $Y$ which dominates $R/P_i$. Let $i_{Y_i}:Y_i\rightarrow Y$ be the natural inclusion. Let $a_i=\ell_{R_{P_i}}(R_{P_i})$.

Let $F_1, F_2$ be Cartier divisors on $Y$ that are supported above $\fm_R$. Then
\begin{equation}\label{N25}
(F_1\cdot F_2)_R
=\sum a_i(i_{Y_i}^*F_1\cdot i_{Y_i}^*F_2)_{R/P_i}
\end{equation}
by  \autoref{AssocForm}. For each $i$, let $Z_i\rightarrow Y_i$ be a birational projective morphism that is a resolution of singularities of $Y_i$, so we have a birational projective morphism 
$$
\prod Z_i\rightarrow \prod Y_i.
$$
Composing with the natural proper morphism $\prod Y_i\rightarrow Y_{\rm red}$, we have that $\prod Z_i\rightarrow Y_{\rm red}$ is a resolution of singularities.

Let $\{E_{i,j}\}$ be the prime  divisors of $Z_i$ which contract to $\fm_{R/P_i}$ in $\Spec(R/P_i)$ and $E(Z_i)$ be the real vector space with basis $\{E_{i,j}\}$. The negative of the intersection product, $-(G,H)_{R/P_i}$ on $Z_i$, induces a positive definite symmetric bilinear form $\phi_i$ on $E(Z_i)\times E(Z_i)$ by  \autoref{CorNegDef}. Thus $\phi=\sum a_i\phi_i$ defines a positive definite form on $(\prod E(Z_i))\times (\prod E(Z_i))$.

Thus for $G_1,G_2\in \prod E(Z_i)$, 
\begin{equation}\label{N26}
\phi(G_1,G_2)^2\ls \phi(G_1,G_1)\phi(G_2,G_2)
\end{equation}
by the Schwartz inequality (see  \cite[p. 580]{Lang}; there is a typo in the proof there: one should take $\alpha=\sqrt{\langle y,y\rangle}$ and $\beta=-\sqrt{\langle x,x\rangle}$).

Let $F_1$ and $F_2$ be  Cartier divisors on $Y$ defined by $I\mathcal O_Y=\mathcal O_Y(F_1)$ and $J\mathcal O_Y=\mathcal O_Y(F_2)$. 
By \autoref{thm:MixedMult},  \autoref{ThmR1}, \autoref{TheoremMultInt},  \autoref{N25}, and \autoref{N26}, 
\begin{align*}
e(I,J)^2=(F_1,F_2)_R^2
&=(\sum a_i(i_{Y_i}^*F_1\cdot i_{Y_i}^*F_2)_{R/P_i})^2\\
&= (\sum a_i(\alpha_i^*i_{Y_i}^*F_1\cdot  \alpha_i^*i_{Y_i}^*F_2)_{R/P_i})^2\\
&=\phi(\{\alpha_i^*i_{Y_i}^*F_1\},\{\alpha_i^*i_{Y_i}^*F_2\})^2\\
&\ls \phi(\{\alpha_i^*i_{Y_i}^*F_1\},\{\alpha_i^*i_{Y_i}^*F_1\})
\phi(\{\alpha_i^*F_2\},\{\alpha_i^*i_{Y_i}^*F_2\})\\
&=
(\sum a_i(\alpha_i^*i_{Y_i}^*F_1,\alpha_i^*i_{Y_i}^*F_1)_{R/P_i})
(\sum a_i(\alpha_i^*i_{Y_i}^*F_2,\alpha_i^*i_{Y_i}^*F_2)_{R/P_i})\\
&=
(\sum a_i(i_{Y_i}^*F_1,i_{Y_i}^*F_1)_{R/P_i})
(\sum a_i(i_{Y_i}^*F_2,i_{Y_i}^*F_2)_{R/P_i})\\
&=(F_1^2)_R(F_2^2)_R=e(I)e(J).
\end{align*}
\end{proof}

The following corollary proves \cite[Theorem 5.4]{ReSh} for excellent local domains.

\begin{corollary}\label{thm:2dim_pos_def}
Let $(R,\fm_R,\kk)$ be a $2$-dimensional excellent local domain. Then the  quadratic form $q^*$ on $G(R)$ is positive definite.
\end{corollary}

\begin{proof}
Let $I,J$ be $\fm_R$-primary ideals of $R$ and 
let $Y\rightarrow\Spec(R)$ be such that $IJ\mathcal O_Y$ is invertible with $Y$ nonsingular. Writing
$
I\mathcal O_Y=\mathcal O_Y(-D)$ and
$J\mathcal O_Y=\mathcal O_Y(-E)$ 
we have that
\begin{align*}
e^*(\overline{I}\,\overline{J}^{-1}
&\mid
\overline{I}\,\overline{J}^{-1})\\
&=
e(I,I)-e(I,J)-e(J,I)+e(J,J)\\
&=-(\mathcal O_Y(-D)^2)_R+(\mathcal O_Y(-D)\cdot \mathcal O_Y(-E))_R+(\mathcal O_Y(-E)\cdot\mathcal O_Y(-D))_R-(\mathcal O_Y(-E)^2)_R\\
&=
-(\mathcal O_Y(E-D)\cdot \mathcal O_Y(E-D))_R,
\end{align*}
where the second equality is by \autoref{thm:MixedMult} and the third one is by   \autoref{CorR18}.
The conclusion now follows from  \autoref{CorNegDef}.
\end{proof}

\begin{remark}
\autoref{thm:2dim_pos_def} raises the natural question of whether in dimension $d\gs 3$ the analogous symmetric multilinear form 
$$e^*(\overline{I}\,\overline{J}^{-1},\ldots, \overline{I}\,\overline{J}^{-1})
=
\sum_{i=1}^d(-1)^{d-i}\binom{d}{i}e(I^{[i]}, J^{[d-i]})
$$
is positive definite. However, already for $d=3$ one observes from 
$$e^*(\overline{I}\,\overline{J}^{-1},\overline{I}\,\overline{J}^{-1},\overline{I}\,\overline{J}^{-1})
=
e(I^{[3]})-3e(I^{[2]},J)+3e(I,J^{[2]})-e(J^{[3]})
$$
that this is not the case, as reversing the roles of $I$ and $J$ changes the sign of the expression.
\end{remark}

Let $(R,\fm_R,\kk)$ be an excellent normal $2$-dimensional local ring. 
We now give a simple geometric  description of  the group $G(R)$. 
Let $f:X\rightarrow \Spec(R)$ be a resolution of singularities and let $H_X$ be the free abelian group on the prime exceptional divisors of $f$ (the prime divisors on $X$ that contract to $\m_R$).
Suppose that $g:Y\rightarrow \Spec(R)$ is another resolution of singularities with associated group $H_Y$ and that there exists a morphism $h:Y\rightarrow X$  factoring $g$ (so that $g=f\circ h$). Then   we have an injective group homomorphism $h^*:H_X\rightarrow H_Y$ defined by $D\mapsto h^*(D)$ for $D\in H_X$, where $h^*(D)$ is defined by $h^*\mathcal O_X(D)\cong \mathcal O_Y(h^*(D))$. 

Given resolutions of singularities $f:X\rightarrow \Spec(R)$ and $g:Y\rightarrow \Spec(R)$, there exists a resolution of singularities $h:Z\rightarrow \Spec(R)$ such that $h$ factors through $Y$ and $X$.

\begin{lemma}\label{lem:unique} Suppose that $f:X\rightarrow \Spec(R)$ and $g:Y\rightarrow \Spec(R)$
are resolutions of singularities. Then there exists at most one morphism $h:Y\rightarrow X$ such that $g=fh$.
\end{lemma}

\begin{proof} Suppose that $h_1:Y\rightarrow X$ and $h_2:Y\rightarrow X$ are two morphisms such that $g=fh_1$ and $g=fh_2$. Let $\alpha$ be the generic point of $X$ and $\beta$ be the generic point of $Y$. We have natural identifications of $\mathcal O_{X,\alpha}$ and $\mathcal O_{Y,\beta}$ with the quotient field $K$ of $R$, so we may regard each $h_i^*$ as being the identity map on $K$. In this way, we can regard the local rings of $X$ and $Y$ as being subrings of $K$, and if $h_i(q)=p$ for $q\in Y$ and $p\in X$, then
$h_i^*:\mathcal O_{X,p}\rightarrow \mathcal O_{Y,q}$ is an inclusion of local rings in $K$ such that 
$\mathcal O_{Y,q}$ dominates $\mathcal O_{X,p}$.

Let $q\in Y$ and $v$ be a valuation of $K$ which dominates $\mathcal O_{Y,q}$. There exists a unique point $p\in X$ such that $v$ dominates $\mathcal O_{X,p}$ since $f$ is proper over $\Spec(R)$ and $v$ dominates $R_{g(q)}$ ($g(q)$ is a prime ideal in $R$). Thus $p=h_1(q)=h_2(q)$, and so the inclusions $h_i^*:\mathcal O_{X,p}\rightarrow \mathcal O_{Y,q}$ are the same maps for all $p\in X$, and so $h_1=h_2$.
\end{proof}

Let 
$$
\mathcal H=\{H_X\mid f_X:X\rightarrow \Spec(R)\mbox{ is a resolution of singularities}\}
$$
with the homomorphisms $h^*:H_X\rightarrow H_Y$ if $h:Y\rightarrow X$ is a morphism such that $f_Y=f_Xh$ (which is unique if it exists by \autoref{lem:unique}).  
 Thus the elements of $\mathcal H$ and the morphisms $h^*:H_X\rightarrow H_Y$ form a directed system of groups.

Let 
$$
H=\lim_{\rightarrow}H_X
$$
be the direct limit of this system.

\begin{proposition}\label{PropG} With the notations above, the map 
$$
\Phi:G(R)\rightarrow H=\lim_{\rightarrow}H_X
$$
defined by 
$$
\Phi(\overline{I}\,\overline{J}^{-1})=[D],
$$
where $D$ is a divisor whose support lies above $\fm_R$ on a resolution of singularities $X\rightarrow \Spec(R)$. 
such that $I\mathcal O_X$ and $J\mathcal O_X$ are invertible 
and $(I\mathcal O_X)(J\mathcal O_X)^{-1}\cong \mathcal O_X(D)$,
is a well-defined group isomorphism.
\end{proposition}

\begin{proof} $\Phi$ is a well-defined group homomorphism since $\Gamma(X,I\mathcal O_X)=\overline{I}$ and $\Gamma(X,J\mathcal O_X)=\overline{J}$.

Suppose that $\Phi(\overline{I}\,\overline{J}^{-1})=0$. There exists a resolution of singularities $X\rightarrow \Spec(R)$ such that $I\mathcal O_X$ and $J\mathcal O_X$ are invertible and 
$(I\mathcal O_X)(J\mathcal O_X)^{-1}\cong \mathcal O_X$. Writing $I\mathcal O_X=\mathcal O_X(D)$ and $I\mathcal O_X=\mathcal O_X(E)$ where the supports of $D$ and $E$ lie above $\fm_R$, we obtain that $D=E$ so that 
$\overline{I}=\Gamma(X,\mathcal O_X(D))=\overline{J}$.
Thus $\overline{I}\,\overline{J}^{-1}$ is the identity of $G(R)$, and so $\Phi$ is injective.

Now we show that $\Phi$ is surjective. Suppose $D\in H_X$ for some $X$.  Let $\{E_1,\ldots, E_r\}$ be the prime  divisors on $X$ that are supported above $\fm_R$. Since the intersection form on $X$ is negative definite  by \cite[Lemma 14.1]{L} (or  \autoref{CorNegDef}),  there exists $C\in H_X$ such that $(C\cdot E_i)>0$ 
for all $i$.  
By \cite[Lemma 14.9]{Bad}, $C=\sum a_iE_i$ with $a_i<0$ for all $i$; that is, $-C$ is effective. 
$\mathcal O_X(C)$ is ample on $X$ by \cite[Proposition 2.3]{CNag}. 
Since $C$ is ample, by standard facts about ample divisors \cite[Exercise II.7.5]{H}, we have that there exits $a>0$ such that $aC$ is very ample,  $aC+D$ is ample, and $-(aC+D)$ is effective. Further, there exists $n_0>0$ such that $n(aC+D)$ is very ample for $n\gs n_0$. Thus $D_1:=(n_0+1)(aC+D)$ and $D_2:=n_0(aC+D)+aC$ are very ample and $-D_1$ and $-D_2$ are effective. 
Let $I=\Gamma(X,\mathcal O_X(D_1))$ and $J=\Gamma(X,\mathcal O_X(D_2))$. $I$ and $J$ are $\fm_R$-primary ideals since $\mathcal O_X(-D_i)\subset \mathcal O_X$ for $i=1,2$ implies $\Gamma(X,\mathcal O_X(D_i))\subset \Gamma(X,\mathcal O_X)=R$ since $R$ is normal.
Then $I\mathcal O_X=\mathcal O_X(D_1)$ and $J\mathcal O_X=\mathcal O_X(D_2)$. 
Thus
$$
(I\mathcal O_X)(J\mathcal O_X)^{-1}\cong\mathcal O_X(D_1-D_2)=\mathcal O_X(D),
$$
and so $\Phi(\overline{I}\,\overline{J}^{-1})=[D]$, showing that $\Phi$ is surjective.
\end{proof}

%%%%%%%%%%%%%%%%%%%%%%%%%%%%
%\vskip 10truein
%%%%%%%%%%%
%%%%%%%%%%%%
%%%%%%%%%%%%

\begin{center}
	{\it Acknowledgments}
\end{center}

We acknowledge support by the NSF Grant  DMS \#1928930, while the authors were in residence at the Simons Laufer Mathematical Science Institute (formerly MSRI) in Berkeley, California, during the Spring 2024 semester.

\bibliographystyle{plain}

\end{document}